\documentclass{article}
\usepackage{amsmath,amsfonts,amssymb,amsthm}
\usepackage{latexsym}
\usepackage{epsfig,overpic}
\usepackage{stmaryrd}
\usepackage{color}
\usepackage{graphicx}
\usepackage{wrapfig}
\usepackage{xcolor}
\usepackage{hyperref}

\newtheorem{lemma}{Lemma}[section]
\newtheorem{example}{Example}
\newtheorem{remark}{Remark}[section]
\newtheorem{theorem}{Theorem}[section]

\providecommand{\keywords}[1]{\textbf{\textit{Keywords:}} #1}

\newcommand{\R}{\mathbb R}

\newcommand{\bU}{\mathbf U}

\newcommand{\bn}{\mathbf n}
\newcommand{\bm}{\mathbf m}
\newcommand{\bp}{\mathbf p}

\newcommand{\bs}{\mathbf s}
\newcommand{\bu}{\mathbf u}
\newcommand{\bv}{\mathbf v}

\newcommand{\blf}{\mathbf f}
\newcommand{\bw}{\mathbf w}
\newcommand{\bx}{\mathbf x}

\newcommand{\T}{\mathcal T}

\newcommand{\Div}{\mbox{\rm div}}

\newcommand{\rd}{\mathrm{d}}

\definecolor{darkred}{rgb}{.7,0,0}

\definecolor{green}{rgb}{0,0.7,0}

\def\MO#1{{\color{black}#1}}

\def\dO{{\partial\Omega} }
\def\dG{{\partial\Gamma} }

\definecolor{linkcolor}{HTML}{799B03} 
\definecolor{urlcolor}{HTML}{799B03} 

\hypersetup{pdfstartview=FitH, linkcolor=linkcolor,urlcolor=urlcolor, colorlinks=true}

\begin{document}

\title{An unfitted finite element method for the Darcy problem in a fracture network}
\author{
Alexey Y. Chernyshenko\thanks{
Institute of Numerical Mathematics, Russian Academy of Sciences, Moscow 119333}
\and
Maxim A. Olshanskii\thanks{Department of Mathematics, University of Houston, Houston, Texas 77204-3008 {\tt molshan@math.uh.edu}}
}

\date{}
\maketitle

\markboth{}{}

\begin{abstract}
The paper develops an unfitted finite element  method for solving the Darcy  system of equations posed
in a network of fractures embedded in a porous matrix. The approach builds on the Hughes--Masud stabilized formulation of the Darcy problem and the trace finite element method.  The system of fractures is  allowed to cut through the background  mesh in an arbitrary way. Moreover, the fractures are not triangulated in the common sense and the junctions of fractures are not fitted by the mesh. To couple the flow variables at multiple fracture junctions, we extend the Hughes--Masud formulation by including  penalty terms to handle interface conditions. One observation made here is that by over-penalizing the pressure continuity interface condition one can avoid including additional jump terms along the fracture junctions. This simplifies the formulation while ensuring the optimal convergence order of the method. The application of the trace finite element allows to treat both planar and curvilinear fractures with the same ease. The paper presents convergence analysis and assesses the performance of the method in a series of numerical experiments. For the background mesh we use an octree grid with cubic cells.
The flow in the fracture can be easily coupled with the flow in matrix, but we do not pursue the topic of discretizing such coupled system here.
 \end{abstract}



\section{Introduction}
Numerical modelling of a flow in a fractured porous medium is a standard problem in geosciences and reservoir simulation~\cite{economides1989reservoir,zoback2010reservoir}. While the literature on this topic is overwhelming (see, e.g., \cite{antonietti2019discontinuous,boon2018robust,chave2018hybrid,flemisch2018benchmarks,Koppel2019} for a snapshot of recent research), the problem of developing an accurate and effective numerical method for a complex network of fractures still constitutes a challenge.
The present paper contributes to the topic by introducing a finite element method for the Darcy problem posed in a system of intersecting  fractures represented by a set of 2D surfaces embedded in a bulk domain. The enabling feature
of the method is that it solely uses the background triangulation of the bulk domain (i.e., a tessellation in simplexes or more general polytopes) which is completely independent of the fracture network. Moreover, it does not require any 2D mesh fitted to the fracture surfaces or their intersections.

Application of geometrically unfitted finite element methods for the modelling of flow and transport in fractured porous medium have been addressed recently in a number of publications; see, e.g., \cite{berrone2013simulations,flemisch2016review,huang2011use}.
Developments most closely related to the approach taken in the present paper are those found in~\cite{chernyshenko2018hybrid,formaggia2014reduced,hansbo2017stabilized}. Thus in \cite{formaggia2014reduced} the authors consider a low order Raviart-Thomas finite element method for the Darcy flow on a 1D network of fractures. Although a triangulation of each fracture surface were build, these triangulations do not match the fracture
intersection points, and the authors applied XFEM methodology to handle discontinuities in the solution over the junctions. The recent paper \cite{fumagalli2019conforming} reviews this and other numerical approaches, where a different degree of the conformity of fracture meshes at junction interfaces is assumed. However, triangulating each fracture branch  can be itself a demanding task for large networks or complex geometry.  Hence, the next level of nonconformity is to abandon the  triangulation of the fracture in the usual sense and to discretize the flow problem along the fracture network only with the help of degrees of freedom tailored to the ambient mesh in the matrix. This ambient (background) mesh should be  independent of the embedded fracture network. This approach was first taken in \cite{ORG09} to discretize scalar elliptic PDEs posed on surfaces and later it evolves to become the Trace FEM methodology~\cite{TraceFEM} and a part of the Cut FEM~\cite{cutFEM}. Trace FEM for the transport and diffusion  of a contaminant in a fractured porous media was recently developed in \cite{chernyshenko2018hybrid}.
For the Darcy problem, the Trace FEM was first studied in \cite{hansbo2017stabilized}.  In that  paper, the authors considered the Darcy problem posed on a surface embedded in a bulk tetrahedra grid. They use a variant of Hughes--Masud week formulation to solve for the pressure and tangential velocity.

Following \cite{hansbo2017stabilized}, we apply Trace FEM in combination with  a variant of Hughes--Masud week formulation. The novelty of the present work is two-fold. First, for the ambient mesh we consider octree Cartesian grids, which can be easily adapted. Second and more importantly, we assume intersecting piecewise smooth surfaces (representing branching fractures), while in the previous work only closed smooth manifolds were considered. The branching leads to discontinuous fluxes and only piecewise smooth pressure field. Handling those without mesh fitting, \emph{but} preserving optimal convergence order, is not straightforward. In the paper, this is achieved by allowing discontinuous velocity and pressure fields in background cells intersected by the fracture junctions and by including a penalty term. This treatment of fracture junctions draw an analogy with the Nitsche-XFEM method of Hansbo and Hansbo \cite{Hansbo} for interface problems and more general with CutFEM~\cite{cutFEM}. One interesting difference, however, is that we use another scaling for the penalty term and skip certain consistency terms (typical for discontinuous Galerkin FEM and Nitsche's method) along the junctions without sacrificing optimal asymptotic  accuracy.
The paper includes both numerical analysis and computational assessment of the method.

The remainder of the paper is organized as follows. Section~\ref{s_setup} defines the mathematical model. Section~\ref{s_FEM} introduces the finite element method. Section~\ref{s_analysis} presents the convergence analysis of the method.
Section~\ref{s_numer} collects the results of several numerical experiments that illustrate the analysis and the performance of the method.  

\section{Mathematical model}\label{s_setup}

Assume  a piecewise smooth surface $\Gamma\subset\Omega$ embedded in the given bulk  domain $\Omega \subset \mathbb{R}^3$.
The surface $\Gamma$ represents a 2D fracture network and consists of several connected components $\overline\Gamma=\cup_{i=1}^N\overline{\Gamma}_i$,
where each $\Gamma_i$ is smooth orientable surface without self-intersections. For the purpose of analysis, we shall assume that each $\Gamma_i$ is a subdomain \MO{of} a larger \MO{$C^2$}--smooth  surface $\widehat{\Gamma}_i$, such that $\partial\widehat{\Gamma}_i\cap\Omega=\emptyset$ and $\partial\Gamma_i$ is piecewise smooth and Lipschitz as a curve in $\widehat{\Gamma}_i$. The individual components $\Gamma_i$ may intersect only by a curve, i.e. $\mbox{meas}_2(\overline{\Gamma}_i\cap\overline{\Gamma}_j)=0$ for  $i\neq j$, and also  ${\Gamma}_i\cap{\Gamma}_j=\emptyset$, for  $i\neq j$ (this condition means that parts of a fracture separated by a junction are treated as different components).
Further $\bn$ is a unit normal vector defined everywhere on $\Gamma$ except junction interfaces. We shall write $\bn_i$ for $\bn$ on   $\Gamma_i$ and similar for other vector and scala fields defined on $\cup_{i=1}^N\Gamma_i$.

Modeling fractures  as 2D interfaces for flow in porous media has been considered in many places in the literature; see, e.g., \cite{alboin2002modeling,angot2009asymptotic,frih2012modeling,martin2005modeling}.
In this framework, the flow along the fracture component ${\Gamma}_i$ is described in terms of tangential velocity field $\bu_i(\bx)$,  having the physical meaning of the flow rate through the cross-section of the fracture, and pressure field $p_i(\bx)$, $\bx \in {\Gamma}_i$.
The steady state flow in $\Gamma$, is governed  by the  Darcy systems
\begin{equation} \label{Darcy}
\left\{
\begin{aligned}
 \MO{K_i^{-1}}\bu_i + \nabla_\Gamma p_i &= \mathbf{f}_i\\
 \Div_\Gamma\bu_i&=g\\
 \bu_i\cdot\bn_i&=0
 \end{aligned}
\right. \quad \text{in}~ \Gamma_i,\quad i=1,\dots,N,
 \end{equation}
together with interface and boundary conditions specified below.
In \eqref{Darcy} and further in the text,  $\nabla_\Gamma$ and  $\Div_\Gamma$  denote the surface tangential gradient and divergence operators; $g$ stands for the source term, which is typically due to the fluid exchange with the porous matrix (not treated in this paper); $\mathbf{f}_i$ is an exterior force per unit area, $\mathbf{f}_i$ is tangential to $\Gamma_i$;  $K_i$ denotes the permeability tensor along the fracture; all $K_i$ are symmetric and \MO{such that for any tangential vector field $\bv$, i.e. $\bv\cdot\bn_i=0$, it holds $\bv^T K_i\bv\ge \zeta_i |\bv|^2$ with some $\zeta_i>0$}, and $\bn^T_{i} K_i\bv=0$. \MO{Hence, $K_i^{-1}\bv$ is well defined for a tangential
field $\bv$.} Note that fracture aperture can be included in $K_i$ by scaling; see, e.g.~\cite{alboin2002modeling}.

When $\Gamma$ is  \textit{piecewise} smooth,  we need further conditions on the edges (fracture junctions).
Consider an edge $e$ shared by  $M_e$  smooth components $\Gamma_{i_k}$, $k=1,\dots,M_e$. Here and in the rest of the paper,
$\{i_k\}_{k=1,\dots,M_e}$ denote the subset of indexes from $\{1,\dots,N\}$, which is specific for each given $e$.
Denote by $\bm_{i}$  the  normal vector on $\dG_i$ in the plane tangential to $\Gamma_i$ and pointing outward.
The conservation of fluid mass yields
\begin{equation}\label{cont_w}
\sum_{k=1}^{M_e}\bu_{i_k}\cdot\bm_{i_k}=0\quad \text{on}~~e.
\end{equation}
The second interface condition is the continuity of pressure over $e$,
\begin{equation}\label{cont_v}
p_{i_1}=\dots=p_{i_{M_e}}\quad \text{on}~~e.
\end{equation}
Denote by $E$ the collections of all fracture junctions. It is reasonable to assume that $E$ is a finite set and $0<\mbox{meas}_1(e)<+\infty$ for any $e\in E$.

Finally, we prescribe the pressure  boundary condition on $\dG_D$ and the flux boundary condition on $\dG_N$, respectively, with
 $\overline{\dG}=\overline{\dG_D}\cup\overline{\dG_N}$,
\begin{equation} \label{bc}
\left\{
\begin{aligned}
\bm_{i}\cdot\bu_i &=\phi_i \quad \text{on}~ \dG_N\cap\dG_i,~~i=1,\dots,N\\
 p&=p_D \quad \text{on}~ \dG_D.
 \end{aligned}
\right.
 \end{equation}

\section{Finite element method}\label{s_FEM}

First we assume a tessellation  $\T_h$  of  the bulk domain $\Omega$ (matrix).
$\T_h$ can be a consistent subdivision into shape-regular tetrahedra. In this paper, we consider Cartesian background mesh with cubic cells. We allow local refinement of the mesh by sequential division
of any cubic cell into 8 cubic subcells. This leads to a grid with an octree hierarchical  structure. This mesh gives the tessellation   $\T_h$ of the bulk domain $\Omega$, $\overline{\Omega}=\cup_{T\in\T_h} \overline{T}$.
We allow the fracture network  $\Gamma\subset\Omega$ to cut through this mesh in an arbitrary way. For the purpose of analysis, we shall assume that the cells cut by $\Gamma$ have a quasi-uniform size with the characteristic size $h$.

Consider now the ambient finite element space of all piecewise trilinear continuous functions with respect to the bulk octree mesh $\mathcal{T}_h$:
\begin{equation}
 V_h:=\{v\in C(\Omega)\ |\ v|_{S}\in Q_1~~ \forall\ S \in\mathcal{T}_h\},\quad\text{with}~~
 Q_1=\mbox{span}\{1,x_1,x_2,x_3,x_1x_2,x_1x_3,x_2x_3,x_1x_2x_3\}.
 \label{e:2.6}
\end{equation}
For every fracture $\Gamma_i$ in the network $\Gamma$ we define the subdomain of $\Omega$ consisting of all cells cut by $\Gamma_i$,
\[
\Omega^i_h=\bigcup\{T\in \T_h\,:\, T\cap\Gamma_i\neq\emptyset\},
\]
and define the restriction of $V_h$ on $\Omega^i_h$, i.e. the space of  piecewise trilinear continuous functions on $\Omega^i_h$,
\begin{equation}
 V_h^{i}:=\{u\in C(\Omega^i_h)\ |\ \exists ~ v\in V_h\  \text{such that }\ u=v|_{\Omega^i_h}\}.
\label{e:fem-space}
\end{equation}
Our trial and test finite element spaces are built from $ V_h^{i}$: We define the pressure space and velocity spaces
\[
Q_h=\bigotimes_{i=1}^N V_h^{i}\quad\text{and}\quad \bU_h=\bigotimes_{i=1}^N [V_h^{i}]^3.
\]
According to the Trace FEM approach,  the finite element solutions of \eqref{Darcy}--\eqref{bc} will be given by \text{traces} of functions from $Q_h$ and $\bU_h$ on $\Gamma$, but the finite element formulation will be written in terms of function defined on  $\bigcup_{i=1}^N\Omega^i_h$. Hence, the method leads to a system of algebraic equations for standard nodal degrees of freedom in the ambient mesh $\T_h$.

Further we use the notation $(\cdot,\cdot)_{Q}$ for the $L^2$ scalar product over a domain $Q$, which can be a 3D, 2D or 1D manifold on different occasions.  For example, with this notation, the Green formula on $\Gamma_i$ reads:
\begin{equation}\label{Green}
(\Div_\Gamma \bv,q)_{\Gamma_i}=-(\bv,\nabla_\Gamma  q)_{\Gamma_i}+(\bm_i\cdot\bv,q)_{\dG_i}
\end{equation}
for any smooth tangential vector field $\bv$ and scalar function $q$ on $\Gamma_i$.

The proposed finite element formulation extends the stabilized mixed formulation for the Darcy problem originally introduced in \cite{masud2002stabilized} for the planar domains.
The key observation here is that the smooth solution to \eqref{Darcy}--\eqref{bc} satisfies the identity
\[
(\MO{K_i^{-1}}\bu+\nabla_\Gamma p,\bv)_{\Gamma_i}+(\Div_\Gamma \bu,q)_{\Gamma_i}+\frac12(\MO{K_i^{-1}}\bu+\nabla_\Gamma p,-\bv+\MO{K_i}\nabla_\Gamma q)_{\Gamma_i} = (g,q)_{\Gamma_i}+\frac12(\blf,-\bv+\MO{K_i}\nabla_\Gamma q)_{\Gamma_i}
\]
for all $q\in H^1(\Gamma_i)$, $\bv\in L^2(\Gamma_i)$ and $i=1,\dots,N$. 
We now set $q=0$ on $\dG_D$ and apply \eqref{Green}. After simple calculations this gives
\begin{equation}\label{aux1}
(\MO{K_i^{-1}}\bu,\bv)_{\Gamma_i}+(\nabla_\Gamma p,\bv)_{\Gamma_i} -(\nabla_\Gamma q,\bu)_{\Gamma_i}+ (\MO{K_i}\nabla_\Gamma p,\nabla_\Gamma q)_{\Gamma_i}+2(\bm_i\cdot\bu,q)_{\dG_i} = 2 (g,q)_{\Gamma_i}+(\blf,-\bv+\MO{K_i}\nabla_\Gamma q)_{\Gamma_i}.
\end{equation}
One further helpful observation is that $p$ and $q$ can be identified with their normal extensions to a neighborhood of $\Gamma_i$ (for each $i$). This identification (which is assumed further in the paper) implies the equality $\nabla_\Gamma p=\nabla p$, which can be further used in \eqref{aux1} to yield
\begin{equation}\label{aux2}
(\MO{K_i^{-1}}\bu,\bv)_{\Gamma_i}+(\nabla p,\bv)_{\Gamma_i} -(\nabla q,\bu)_{\Gamma_i}+ (\MO{K_i}\nabla p,\nabla q)_{\Gamma_i}+2(\bm_i\cdot\bu,q)_{\dG_i} = 2 (g,q)_{\Gamma_i}+(\blf,-\bv+\MO{K_i}\nabla q)_{\Gamma_i}.
\end{equation}
This corresponds to so-called \emph{full gradient} formulation of the surface PDES; see \cite{Alg2,reusken2015analysis}. \MO{The full gradient formulation exploits the embedding of $\Gamma$ in the ambient Euclidian space and, in general, provides extra stability for a finite element method based on external elements. The formulation is consistent for any ambient finite element method, which aims to approximate the surface solution together with its normal extension.}  In the context of the surface Darcy problem, the full gradient formulation was used in~\cite{hansbo2017stabilized}. We finally sum up equalities \eqref{aux2} for all $i=1,\dots,N$ and use the interface condition \eqref{cont_w} and the boundary condition for fluxes from \eqref{bc} to conclude that any smooth solution of \eqref{Darcy}--\eqref{bc} satisfies
\begin{multline}\label{week}
(\MO{K^{-1}}\bu,\bv)_{\Gamma}+(\nabla p,\bv)_{\Gamma} -(\nabla q,\bu)_{\Gamma}+ (\MO{K_i}\nabla p,\nabla q)_{\Gamma}+\sum_{e\in E}\frac2{M_e}\sum_{k=1}^{M_e-1}\sum_{\ell={k+1}}^{M_e}(\bm_{i_k}\cdot\bu_{i_k}-\bm_{i_\ell}\cdot\bu_{i_\ell},q_{i_k}-q_{i_\ell})_{e} \\ = 2 (g,q)_{\Gamma}+(\blf,-\bv+\MO{K_i}\nabla q)_{\Gamma}-2(\psi,q)_{\dG_N}
\end{multline}
for any $q\in \bigotimes_{i=1}^N H^1(\Gamma_i)$ such that $q=0$ on $\dG_D$ and $\bv\in L^2(\Gamma)^3$. To handle the sum of the edge terms, we used \eqref{cont_w} and the identity
\[
\sum_{i=1}^{M}a_ib_i=\frac1M\left(\big(\sum_{i=1}^{M}a_i\big)\big(\sum_{i=1}^{M}b_i\big)+\sum_{i=1}^{M-1}\sum_{j={i+1}}^{M}(a_i-a_j)(b_i-b_j)\right)
\]
for any $a_i,b_i\in\mathbb{R}$.

Our finite element method is based on the equality \eqref{week}. \MO{Note that we may}  assume that $K_i$ is extended to be symmetric positive definite in $\mathbb{R}^3$, \MO{rather than only on the tangential space, since this does not affect any quantities in \eqref{week}, but would be helpful, when we proceed with the finite element formulation.}
To approximate pressure, we use finite element functions from $Q_h$, which are discontinuous across $e\in E$. Therefore, we add a penalty term to our formulation to weekly enforce the pressure continuity condition from \eqref{bc}. Furthermore, we \emph{over-penalize} this condition, by choosing a different scaling of the penalty parameter compared to the standard Nitsche's~\cite{Hansbo} or discontinuous Galerkin methods~\cite{arnold2002unified}. It turns out that the over-penalization allows one to skip other edge terms in the finite formulation. This greatly simplifies the method while keeping the consistency order optimal.
Summarizing, the finite element method reads:
 Find $\bu_h\in \bU_h$ and $p_h\in Q_h$ such that $p_h|_{\partial\Gamma_{D}}=\MO{I_h^b}(p_D)$ and
\begin{equation}\label{FEM0}
  a(\bu_h,p_h;\bv_h,q_h)=f(\bv_h,q_h)
\end{equation}
for all $\bv_h\in \bU_h$ and $q_h\in Q_h$ such that $q_h|_{\partial\Gamma_{D}}=0$, with
\begin{align}
a(\bu,p;\bv,q)&=(\MO{K^{-1}}\bu,\bv)_{\Gamma}+(\nabla p,\bv)_{\Gamma} -(\nabla q,\bu)_{\Gamma}+ (\MO{K_i}\nabla p,\nabla q)_{\Gamma} +\underbrace{\sum_{e\in E}\frac{\rho_e}{h^2}\sum_{k=1}^{M_e-1}\sum_{\ell=k}^{M_e}(p_{i_k}-p_{i_\ell},q_{i_k}-q_{i_\ell})_{e}}_{\text{penalty term to enforce pressure continuity}} \notag
\\
 &\qquad+\underbrace{\sum_{i=1}^N \rho_u h(\bn_i\cdot\nabla \bu_i ,\bn_i\cdot\nabla  \bv_i)_{\Omega_h^i}+
 \sum_{i=1}^N \rho_p h(\bn_i\cdot\nabla p_i,\bn_i\cdot\nabla  q_i)_{\Omega_h^i}}_{\text{normal volume stabilization}}, \label{FEM} \\
 f(\bv,q)&= 2 (g,q)_{\Gamma}+(\blf,-\bv+\MO{K_i}\nabla q)_{\Gamma}-2(\psi,q)_{\dG_N}. \notag
\end{align}
 Here $\rho$'s are tunable parameters,
which we set (in both analysis and experiments) to be equal to $1$; \MO{$I_h^b(p_D)$ is the interpolation of the boundary condition, which we define by extending pressure values from $\partial\Gamma_{D}$ along normal directions in $\dO$ to the corresponding  nodal values  from $\Omega_h^i\cap\dO$.}

\begin{remark}[Normal volume stabilization]\label{rem_stab}\rm
We briefly discuss the  ``normal volume stabilization'' terms in \eqref{FEM}.
The term involve the extension of the normal vector to $\Omega_h^i$, which can be defined as $\bn_i(\bx)=\nabla \mbox{dist}(\bx,\widehat{\Gamma}_i)$. Assuming that the mesh is fine enough to resolve the (curvilinear) geometry, this definition  gives the meaning to the normal vector in all $\Omega_h^i$, including mesh cells cut by $\dG_i$.
Next, we note that the normal volume stabilization terms vanish for the solution $\bu$, $p$ of the Darcy equations~\eqref{Darcy}, because we assume the normal extension of the solution off the fracture components.
Finally, these terms are included,  following~\cite{burman2016cutb,grande2017analysis}, to ensure that algebraic properties of the resulting linear systems are insensitive to the position of $\Gamma$ against the background mesh.
Indeed, if $\rho_u=0$ or $\rho_p=0$, then for a natural nodal basis in $\bU_h$ and $Q_h$, small cuts of the background elements by the surface  may lead to arbitrarily small diagonal entries in the resulting matrix. The stabilization terms in \eqref{FEM} eliminate this problem since for the choice of $\rho_u=O(1)$ and $\rho_p=O(1)$ they allow to get control over the $L^2(\Omega_h^i)$-norms of $\bv_h\in (V_h^i)^3$, $q_h\in V_h^i$ by the problem induced norms. The analysis of this acquired  algebraic stability  can be found at several places in the literature, e.g.  \cite{burman2016cutb,grande2017analysis} for the Laplace--Beltrami problem or
\cite{olshanskii2018finite} for the surface Stokes problem, so we omit repeating it here.
\end{remark}

\begin{remark}[Overpenalty]\rm In the framework of discontinuous Galerkin (DG) methods, a  technique similar to the overpenalty used to enforce pressure continuity here is known as a superpenalty; see \cite{arnold2002unified} and references there. Compared to the  superpenalty technique in DG FEM, we have a weaker dependence of the penalty parameter on the negative power of $h$, which is beneficial for the condition number of the resulting matrices. On the analysis size, the superpenalty DG method exploits  the availability of a continuous finite element interpolant across the element edges, which is not the case here. \MO{As a consequence, to show a suitable consistency bound, we have to apply a different argument comparing to the analysis of the superpenalty DG method. This results in the extra smoothness assumption for pressure solution, e.g. $p$ is from $H^3$ on every $\Gamma_i$ rather than from $H^2$.}
\end{remark}

\begin{remark}[Internal parts of $\partial\Gamma$]\label{Rem_Bound}\rm For the finite element formulation in \eqref{FEM} we assumed  that $\dG\subset\dO$.  If  $\partial\Gamma_D$ has a part strictly inside $\Omega$, then the pressure boundary condition can be enforced by including additional penalty term of the form
\[
\sum_{i=1}^{N}\frac{\rho_i}{h^2}(p_i-p_D,q_i)_{\Gamma_i\cap\Gamma_D}
\]
to the finite element formulation \eqref{FEM}. Internal boundaries with prescribed fluxes, i.e. $\partial\Gamma_N\subset\Omega$, do not affect the formulation in \eqref{FEM}.
\end{remark}

\subsection{Numerical integration}
The finite element formulation \eqref{FEM} requires computing surface integrals. If $\Gamma_i$ is a planar component, then
numerical integration is straightforward.   For a  curvilinear $\Gamma$, in general, we need to know a (local) parametrization of the surface to compute integrals in \eqref{FEM}. For implicitly given surfaces (for example, for surfaces defined as the zero of a distance function), the numerical integration is a more subtle issue; see, e.g., \cite{olshanskii2016numerical}.
In the present paper, for  numerical tests with curvilinear surfaces we compute surface integrals by using a polygonal second order approximation of $\Gamma_i$, denoted by $\Gamma_{h,i}$.
We construct $\Gamma_{h,i}$ as follows. For  $\Gamma_i$ let $\phi$ be a Lipshitz-continuous level set function, such that $\phi(\bx)=0$ on $\Gamma_i$. We set $\phi_h=I(\phi)$,  a nodal interpolant of $\phi$ by a  piecewise trilinear continuous function with respect to the octree grid $\mathcal{T}_h$.  Further, consider the zero level set of $\phi_h$,
$
\widetilde{\Gamma}_{h,i}:=\{\bx\in\Omega\,:\, \phi_h(\bx)=0 \}.
$
If $\Gamma_i$  is smooth, then  $\widetilde{\Gamma}_{h,i}$ is an approximation to $\Gamma_{i}$ in the following sense:
\begin{equation}\label{eq_dist}
\mbox{dist}(\Gamma_{i},\Gamma_{h,i})\le ch^2_{\rm loc},\qquad |\bn(\bx)-\bn_h(\tilde\bx)|\le c h_{\rm loc},
\end{equation}
where $\bx$ is the closest point on $\Gamma_i$ for $\tilde\bx\in\widetilde{\Gamma}_{h,i}$ and $h_{\rm loc}$ is the local mesh size.
We note that in some applications, $\phi_h$ is computed from a solution of a discrete indicator function equation, without a direct knowledge of $\Gamma$.

\begin{figure}
  \begin{center}a)
    \includegraphics[width=0.31\textwidth]{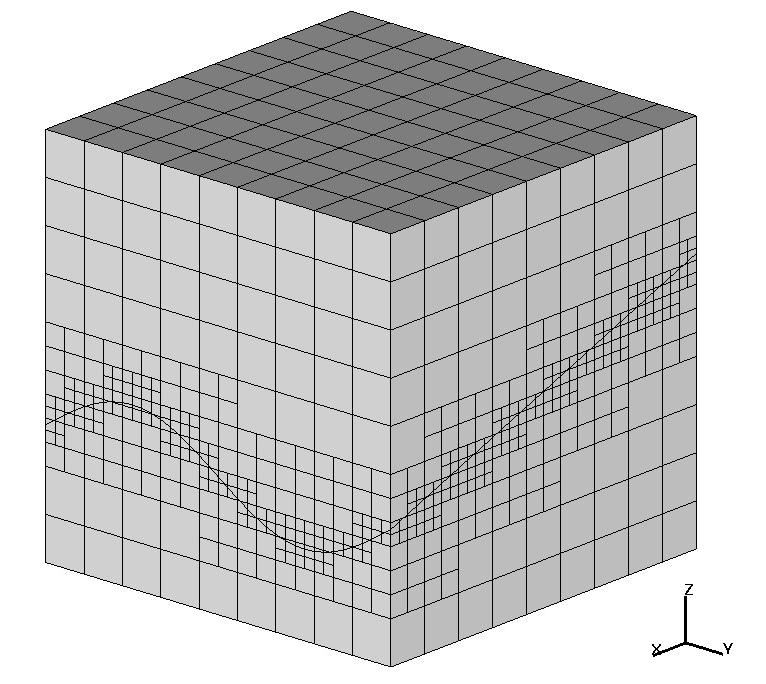}b)
    \includegraphics[width=0.31\textwidth]{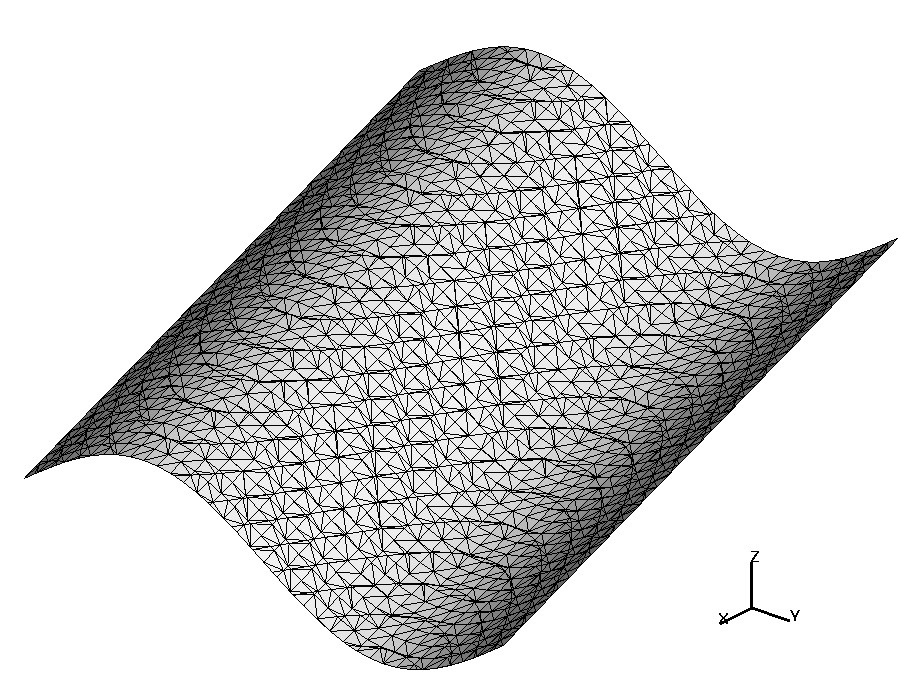} c)
    \includegraphics[width=0.29\textwidth]{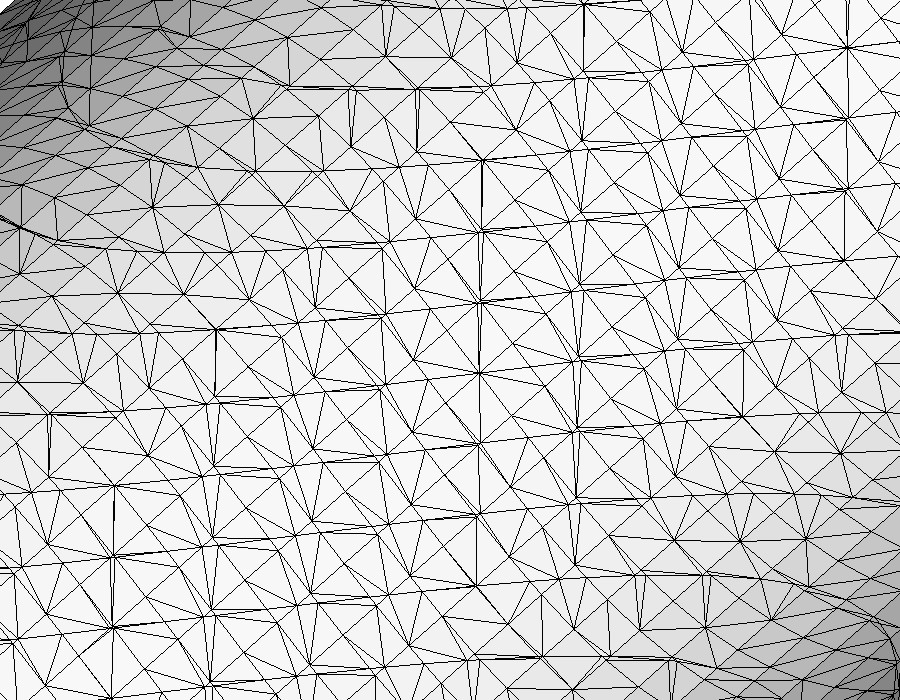}
    \caption{a) Example of a  bulk domain with one fracture. In this example, the background mesh is refined near the fracture; b) The reconstructed $\Gamma_h$; c) The zoom-in of the induced surface triangulation used for numerical integration.}
    \label{fig1}
  \end{center}
\end{figure}
Once $\phi_h$ is computed, we recover $\Gamma_{h,i}$ by the cubical marching squares method from \cite{MCM2} (a variant of the very well-known marching cubes method). The method provides a triangulation of $\widetilde{\Gamma}_h$ within each
cube such that the global triangulation is continuous, the number of triangles within each cube is finite and bounded
by a constant independent of $\widetilde{\Gamma}_{h,i}$ and a number of refinement levels. Moreover, the vertices of triangles from $\mathcal{F}_h$ are lying on $\widetilde{\Gamma}_{h,i}$. This final discrete surface $\Gamma_{h,i}$ is still an approximation of $\Gamma_i$ in the sense of \eqref{eq_dist}. An example of bulk domain with embedded surface and background mesh is illustrated in Figure~\ref{fig1}. Note that the resulting ``triangulation'' of $\Gamma_{h,i}$ is \emph{not} shape  regular. This is not a problem, since this triangulation is used only to define quadratures in the finite element method, while approximation properties of the method depend on the volumetric octree mesh.

\section{Error analysis} \label{s_analysis}

The error analysis fits the standard finite element framework. Certain care and less standard arguments will be needed to show optimal order consistency for the formulation in \eqref{FEM}. Interpolation results rely on   approximation properties of the polynomial traces on smooth surfaces. We start with the definition of the norms used further and the proof of numerical stability.

\subsection{Stability}
For  analysis, we need the broken Sobolev spaces
\[
Q=\bigotimes_{i=1}^N H^1(\Omega_h^i)\cap H^1(\Gamma_i) \quad\text{and}\quad \bU=\bigotimes_{i=1}^N [H^1(\Omega_h^i)]^3,\qquad Q_h\subset Q,\quad \bU_h\subset\bU.
\]
The subspaces of functions from $Q$ and $Q_h$ vanishing on $\dG_D$ are denoted by $Q^0$ and $Q_h^0$, respectively.
The stability estimate  of the method  involves the following problem-dependent velocity and pressure norms:
\[
\|\bv\|_\ast^2= (\MO{K^{-1}}\bv,\bv)_{\Gamma} +\sum_{i=1}^N \rho_u h\|\bn_i\cdot\nabla  \bv_i\|^2_{\Omega_h^i},
\quad
\|q\|_\ast^2= \|\MO{K}\nabla q\|^2_{\Gamma}+ \sum_{e\in E}\frac{\rho_e}{h^2}\sum_{k=1}^{M_e-1}\sum_{\ell=k}^{M_e}\|q_{i_k}-q_{i_\ell}\|^2_{e}+
 \sum_{i=1}^N \rho_p h\|\bn_i\cdot\nabla  q_i\|^2_{\Omega_h^i},
\]
where $\bv\in\bU$, $q\in Q$. For the later expression to define a norm on $Q^0$ we shall assume that $\mbox{meas}_1(\dG_D)>0$ and $\Gamma$ is connected.
Otherwise, if $\mbox{meas}_1(\dG_D)=0$, one uses a factor norm, since the pressure solution to \eqref{Darcy}--\eqref{bc} is defined only up to the addition of the hydrostatic constant mode. On the product space $\bU\times Q$, we define
\[
\|\bv,q\|=(\|\bv\|_\ast^2+\|q\|_\ast^2)^{\frac12}.
\]

Given the definitions above, one immediately checks the coercivity and continuity bounds for the finite element method bilinear form,
\begin{equation}\label{coer}
\|\bv,q\|^2\le a(\bv,q;\bv,q)\quad \forall\, \bv\in\bU,\, q\in Q,
\end{equation}
and
\begin{equation}\label{contA}
a(\bu,p;\bv,q)\le 2 \|\bu,p\|\|\bv,q\|\quad \forall\, \bu,\bv\in\bU,\, p,q\in Q.
\end{equation}
\smallskip

The repeating application of the Sobolev inequality
\[
\|q\|_{\Gamma_i} \le C(\Gamma_i,e)(\|\nabla_\Gamma q\|_{\Gamma_i}+\|q\|_e)\quad q\in H^1(\Gamma_i),\, e\in\dG_i,\, \mbox{meas}_1(e)>0,
\]
 and the trace inequality
\begin{equation}\label{Trace2}
\|q\|_e  \le C(\Gamma_i,e)(\|\nabla_\Gamma q\|_{\Gamma_i}+\|q\|_{\Gamma_i})\quad q\in H^1(\Gamma_i),\, e\in\dG_i,\, \mbox{meas}_1(e)>0,
\end{equation}
leads to the Poincare inequality on $\Gamma$:
\begin{equation}\label{Poinc}
 \|q\|_\Gamma\le  C \left(\|\nabla q\|^2_{\Gamma}+ \sum_{e\in E}\frac{\rho_e}{h^2}\sum_{k=1}^{M_e-1}\sum_{\ell=k}^{M_e}\|q_{i_k}-q_{i_\ell}\|^2_{e}\right)^{\frac12}\le C\|q\|_\ast\quad \forall\, q\in Q^0,
\end{equation}
where $C$ depends on $\Gamma$, $\dG_D$ and the permeability tensor $K$.

With the help of the Poincare, Cauchy-Schwartz, trace and triangle inequalities, one shows
\begin{equation*}
\begin{split}
  f(\bv,q)&\le 2 \|g\|_{\Gamma}\|q\|_{\Gamma}+(\|K^{\frac12}\blf\|_{\Gamma}\|K^{-\frac12}\bv\|_{\Gamma}+\|K^{\frac12}\blf\|_{\Gamma}\|K^{\frac12}\nabla q\|_{\Gamma})
 -2\|\psi\|_{\dG_N} \|q\|_{\dG}\\
 &\le C\left((\|g\|_{\Gamma}+\|\psi\|_{\dG_N}+ \|K^{\frac12}\blf\|_{\Gamma}) \|q\|_{\ast}+ \|K^{\frac12}\blf\|_{\Gamma} \|\bv\|_{\ast}\right).
\end{split}
\end{equation*}
Finally, from \eqref{FEM0}, \eqref{coer} and the estimate above we get the stability bound for the finite element solution
\begin{equation}\label{Stab}
\|\bu_h\|_\ast+\|p_h\|_\ast\le C\left(\|g\|_{\Gamma}+\|\psi\|_{\dG_N}+ \|\blf\|_{\Gamma}\right),
\end{equation}
where $C$ depends on $\Gamma$, $\dG_D$ and $K$, but not on the position of $\Gamma$ over the background mesh. \MO{Penalty and volumetric stabilization terms (cf. Remark~\ref{rem_stab}) in the definition of the norms on the left-hand side depend on $h$.}

\subsection{Error bound}
For the error analysis, we shall assume that all fractures are planar, so that $\widehat{\Gamma}_i\simeq \mathbb{R}^2$ is just a plane. We believe that the error estimate below still holds for fractures with non-zero curvatures, but the analysis \MO{needs  trace } and extensions results for functions defined on submanifolds \MO{as in \eqref{Ext1} and \eqref{stripEst}}, which we do not find in the literature and including their proof would made this paper excessively technical.

As usual, the error analysis needs some extra regularity of the solution, namely $\bu_i\in H^1(\Gamma_i)^3$ and $p_i\in H^3(\Gamma_i)$ for each $i=1,\dots,N$.
We first observe that the finite element formulation \eqref{FEM0} is consistent up to missing interface integrals. To see this, we give sense to the solution components $\bu_i$ and $p_i$ not just on $\Gamma_i$ but in the neighborhood $\Omega^i_h$. To this end, we first consider Stein's~\cite{stein2016singular} extensions of $\bu_i$, $p_i$ to some $E\bu_i\in H^1(\widehat\Gamma_i)^3$ and $Ep_i\in H^3(\widehat\Gamma_i)$ such that
\begin{equation}\label{Ext1}
E\bu_i=\bu_i~\text{on}~\Gamma_i,~~Ep_i=p_i~\text{on}~\Gamma_i\quad\text{and}\quad
\|E\bu_i\|_{H^1(\widehat\Gamma_i)}\le C\|\bu_i\|_{H^1(\Gamma_i)},~~
\|E p_i\|_{H^3(\widehat\Gamma_i)}\le C\|p_i\|_{H^3(\Gamma_i)}
\end{equation}
with some finite $C$ depending only on $\Gamma_i$. 
Now we define normal extensions in $\Omega^i_h$, $p^e_i=Ep_i\circ\bp$, where $\bp$ is the closest point projection on  $\widehat\Gamma_i$,  and similar we define $\bu_i^e$. From the properties of the normal extension 
we get $p^e_i\in H^3(\Omega_h^i)$, $\bu_i^e\in H^1(\Omega_h^i)^3$ and for the norms it holds $\|p_i^e\|^2_{H^3(\Omega_h^i)}\le C\,h\,\|Ep_i\|^2_{H^3(\widehat{\Gamma}_i)}$,   $\|\bu_i^e\|^2_{H^1(\Omega_h^i)}\le C\,h\,\|E\bu_i\|^2_{H^1(\widehat{\Gamma}_i)}$; see, e.g., \cite{reusken2015analysis}. Combining this with \eqref{Ext1} gives
\begin{equation}\label{ExtBound}
\|\bu_i^e\|_{H^1(\Omega_h^i)}\le C\,h\,\|\bu_i\|_{H^1(\Gamma_i)},~~
\|p_i^e\|_{H^2(\Omega_h^i)}\le C\,h\,\|p_i\|_{H^2(\Gamma_i)},
\end{equation}
with some finite $C$ depending only on $\Gamma_i$. If no confusion arises, we further identify $\bu$ and $p$ with their extensions defined above.

For handling  fracture junctions terms, we also consider the $h$-neighborhood of $\partial\Gamma_i$ in $\widehat{\Gamma}_i$, denoted by $\mathcal{O}(\partial\Gamma_i)$. Then it holds (see, Lemma~4.10 in \cite{ER2013})
\begin{equation}\label{stripEst}
\|v\|_{\mathcal{O}(\partial\Gamma_i)}^2\le C\,h\,\|v\|_{H^1(\widehat{\Gamma}_i)}^2\quad\text{for}~v\in H^1(\widehat{\Gamma}_i).
\end{equation}
We can always assume that $\mathcal{O}(\partial\Gamma_i)$ is wide enough so that $T\cap\widehat\Gamma_i\subset\mathcal{O}(\partial\Gamma_i)$ for all $T\in \T_h$ such that $T\cap\partial\Gamma_i\neq\emptyset$.

Now the normal volume stabilization terms make sense (and equals zero) for $\bu$ and $p$, and we see that  the piecewise smooth solution $\bu,p$ to the network Darcy problem \eqref{Darcy}--\eqref{bc} satisfies  the equality
\begin{equation}\label{consist}
a(\bu,p;\bv,q)=f(\bv,q)-\mbox{E}(\bu;q)\quad\forall\, \bv\in \bU,\,q\in Q,
\end{equation}
with
\[
\mbox{E}(\bu;q)=
\sum_{e\in E}\frac2{M_e}\sum_{k=1}^{M_e-1}\sum_{\ell=k}^{M_e}(\bm_{i_k}\cdot\bu_{i_k}-\bm_{i_\ell}\cdot\bu_{i_\ell},q_{i_k}-q_{i_\ell})_{e}.
\]
Using the Cauchy--Schwartz and triangle inequalities, the definition of the $\|\cdot\|_\ast$ norm on $Q$ and $|\bm_i|=1$, one readily checks the upper bound,
\begin{equation}\label{contE}
\mbox{E}(\bu;q)\le C\, \left(\sum_{e\in E}\sum_{k=1}^{M_e}h^2\|\bu_{i_k}\|_e^2 \right)^{\frac12}\|q\|_\ast,
\end{equation}
with a constant $C$ depending only on $\Gamma$.
\smallskip

We proceed with the interpolation bounds.
\begin{lemma}
Let $\bu\in\bigotimes_{i=1}^N H^1(\Gamma_i)^3$,  $p \in \bigotimes_{i=1}^N H^3(\Gamma_i)$.  Assume $h\le h_0$, where $h_0$ may depend on $\Gamma$, then it holds
\begin{equation}\label{interp}
\inf\limits_{\bw_h\in \bU_h,\,\xi_h\in Q_h}
\|\bu - \bw_h,p - \xi_h\| \le C h(\|\bu \|_{1} +\|p\|_{3}),
\end{equation}
with a constant $C$ independent on how $\Gamma$ intersects with $\T_h$.
\end{lemma}
\begin{proof} The analysis of the interpolation properties of the Trace FEM is commonly based on the local trace inequality, see e.g. \cite{reusken2015analysis,burman2016cutb}, which in our case takes the form:
\begin{equation}\label{eq_trace}
\|v\|_{L^2(T\cap \widehat\Gamma_i)}^2\le C( h^{-1}_T\|v\|_{L^2(T)}^2+h_T\|\nabla v\|_{L^2(T)}^2)\quad \forall~v\in H^1(T)~~ \text{and}~ T\in\Omega_h^i,
\end{equation}
where $h_T=\text{diam}(T)$ and $C$ is independent of $T$, $v$ and how $ \widehat\Gamma_i$ cuts through $T$. A quick proof of \eqref{eq_trace}  consists in dividing the cubic cell $T$ into a finite number of regular tetrahedra
and further applying Lemma~4.2 from \cite{Hansbo} on each of these tetrahedra. For handling edge terms, we need the extension of \eqref{eq_trace} for \emph{curves} cutting through the mesh. More precisely, we need the following inequality:
\begin{equation}\label{eq_trace2}
\|v\|_{L^2(T\cap \partial\Gamma_i)}^2\le C( h^{-2}_T\|v\|_{L^2(T)}^2+\|\nabla v\|_{L^2(T)}^2+h_T^2\|D^2 v\|_{L^2(T)}^2)\quad \forall~v\in H^2(T)~~ \text{and}~ T\in\Omega_h^i,~T\cap \partial\Gamma_i\neq\emptyset,
\end{equation}
where $C$ is independent of $T$, $v$ and how $ \partial\Gamma_i$ cuts through $T$. We provide the proof of \eqref{eq_trace2} in Appendix.

We recall that $\bu_i\in H^1(\Gamma_i)^3$ and $p_i\in H^3(\Gamma_i)$ are identified with their extensions to $\Omega_h^i$ such that \eqref{ExtBound} holds.
Let  $\bw_h=I_h\big(\bu\big)\in\bU_h$, $q_h=I_h(p)\in Q_h$ be the  finite element (Clement) interpolants.
Let us first treat the  edge term in the definition of the $\|p-q_h\|_\ast$:
Let  $e\subset\partial\Gamma_i$ for $e\in E$.
Using interpolation properties of bilinear polynomials and \eqref{eq_trace2},  we  have
for any $T\in\Omega_h^i,~T\cap e\neq\emptyset$:
\begin{equation*}
 \|p-q_h\|_{L^2(T\cap e)}^2\le C( h^{-2}_T\|p-q_h\|_{L^2(T)}^2+\|\nabla (p-q_h)\|_{L^2(T)}^2+h_T^2\|D^2 (p-q_h)\|_{L^2(T)}^2)\le C h_T^2\|p\|_{H^2(T)}^2.
\end{equation*}
Summing up the above inequality over all $T$ intersecting $e$ (the domain formed by all such cells is denoted by $\widetilde{\Omega}_h^i$), we get
\begin{equation}\label{aux5}
 \|p-q_h\|_{L^2(e)}^2\le  C h^2\|p\|_{H^2(\widetilde{\Omega}_h^i)}^2 \le  C h^3\|p\|_{H^2(\mathcal{O}(\partial\Gamma_i))}^2\le  C h^4\|p\|_{H^3(\Gamma_i)}^2.
\end{equation}
Here we used \eqref{ExtBound} (which remains true with $\widetilde{\Omega}_h^i$ and $\mathcal{O}(\partial\Gamma_i)$ instead of ${\Omega}_h^i$ and $\Gamma_i$) and \eqref{stripEst}.
For the rest of $\|p-q_h\|_\ast$ we use interpolation properties of bilinear polynomials, \eqref{eq_trace} and \eqref{ExtBound} to obtain
\begin{equation}\label{aux5.5}
\begin{split}
\|\MO{K_i}\nabla (p-q_h)\|^2_{\Gamma_i}&\le C\sum_{T\in\Omega_h^i}( h^{-1}_T\|\nabla (p-q_h)\|_{L^2(T)}^2+h_T\|D^2 (p-q_h)\|_{L^2(T)}^2)\\
&
\le C\sum_{T\in\Omega_h^i}h_T\|p\|_{H^2(T)}^2 = C h\|p\|_{H^2(\Omega_h^i)}^2 \le C h^2\|p\|_{H^2(\widehat\Gamma_i)}^2 \le C h^2\|p\|_{H^2(\Gamma_i)}^2.
\end{split}
\end{equation}
and
\begin{equation}\label{aux6}
\rho_p h\|\bn_i\cdot\nabla  (p-q_h)\|^2_{\Omega_h^i} 
\le C h^3\|p\|_{H^2(\Omega_h^i)}^2 \le C h^4\|p\|_{H^2(\widehat\Gamma_i)}^2\le C h^4\|p\|_{H^2(\Gamma_i)}^2.
\end{equation}
Estimates \eqref{aux5}--\eqref{aux6} lead to the desired bound on $\|p-q_h\|_\ast$.
Similar to \eqref{aux5.5}--\eqref{aux6}, we use interpolation properties of bilinear polynomials, \eqref{eq_trace} and \eqref{ExtBound}, to obtain
\[
\|\bu-\bw_h\|_\ast \le C h \sum_{i=1}^{N}\|\bu\|_{H^1(\Gamma_i)} \le C h \|\bu \|_{1},
\]
with a constant $C$ independent of $h$ and how $\Gamma$ intersects the background mesh.\\
\end{proof}

Since for each component $\Gamma_i$ we associate with $\bu_i$ and $p_i$ there  extensions to $\Omega_h^i$,  the error functions $\bu-\bu_h$ and $p-p_h$ are well-defined as functions of $\bU$ and $Q$. We are ready to prove the following convergence result.

\begin{theorem} \label{thm2}  Let $(\bu,p)$ be the solution of \eqref{Darcy}--\eqref{bc} and assume that $\bu\in\bigotimes_{i=1}^N H^1(\Gamma_i)^3$,  $p \in \bigotimes_{i=1}^N H^3(\Gamma_i)$. Let  $(\bu_h,p_h) \in \bU_h\times Q_h$ be the solution  of \eqref{FEM0}. The following discretization error bound holds:
\begin{equation}\label{discrbound}
 \|\bu-\bu_h,p  - p_h \|  \le C h(\|\bu \|_{1} + \|p\|_{3}).
\end{equation}
Here $\|\cdot\|_k$, $k=1,3$ denotes the broken Sobolev spaces norms for $\bigotimes_{i=1}^N H^k(\Gamma_i)$; the constant $C$ depends on $\Gamma$, but not on how $\Gamma$ intersects the background mesh.
\end{theorem}
\begin{proof}
Using the coercivity  and consistency properties in \eqref{coer} and \eqref{consist} as well as continuity estimates for the $a$ and $\mbox{E}$ forms \eqref{contA}, \eqref{contE},  we obtain, for arbitrary $(\bw_h,\xi_h) \in \bU_h \times Q_h$:
\begin{align*}
 \|\bu_h- \bw_h,p_h- \xi_h\|^2 &\le a(\bu_h- \bw_h,p_h- \xi_h;\bu_h- \bw_h,p_h- \xi_h)  \\
 & = a(\bu- \bw_h,p- \xi_h;\bu_h- \bw_h,p_h- \xi_h)-\mbox{E}(\bu;p_h- \xi_h)\\
 &\le C\|\bu - \bw_h,p - \xi_h\| \|\bu_h - \bw_h,p_h - \xi_h\| + C\, \left(\sum_{e\in E}\sum_{k=1}^{M_e}h^2\|\bu_{i_k}\|_e^2 \right)^{\frac12}\|p_h- \xi_h\|_\ast.
\end{align*}
Therefore, after cancellation and using the trace inequality \eqref{Trace2} we get
\begin{equation}\label{aux4}
\begin{split}
\|\bu_h- \bw_h,p_h- \xi_h\|&\le C\left(\|\bu - \bw_h,p - \xi_h\| + \left(\sum_{e\in E}\sum_{k=1}^{M_e}h^2\|\bu_{i_k}\|_e^2 \right)^{\frac12} \right)\\
&\le C\left(\|\bu - \bw_h,p - \xi_h\|+ h \sum_{i=1}^{N}\|\bu\|_{H^1(\Gamma_i)}\right).
\end{split}
\end{equation}
For  $(\bw_h,\xi_h)\in \bU_h \times Q_h$ we take optimal finite element interpolants for the (normal extensions of the) solution  $\bw_h=I_h\big(\bu\big)$, $q_h=I_h(p)$.
Now, the triangle inequality, \eqref{aux4} and \eqref{interp} leads to  \eqref{discrbound}:
\[
\|\bu-\bu_h,p  - p_h \|\le \|\bu_h- \bw_h,p_h- \xi_h\|+\|\bu - \bw_h,p - \xi_h\|\le  C h (\|\bu \|_{1} +\|p\|_{2}).
\]
\end{proof}

\begin{remark}[$O(h^2)$ convergence] \rm
In \cite{hansbo2017stabilized} the trace finite element method as in \eqref{FEM} applied on a smooth closed surface $\Gamma$ was  proved to enjoy higher convergence in weaker norms for the pressure and fluxes.  In our setting, this would mean the estimate
\begin{equation}\label{HigherOrder}
\|p- p_h\|_{\Gamma}\le C\,h\,\|\bu-\bu_h,p  - p_h \|=O(h^2)
\end{equation}
for the pressure error. The convergence leverage argument, as usual, is based on the $H^2$- regularity estimate for the solution of the dual problem, which is the same system of Darcy equations in the fracture network in our case.
However, we are not aware of a suitable regularity results for the case of intersecting fractures. Namely, we would need
the estimate of the norm $\|p\|_{2}+\|\bu\|_{1}\le C \|\blf\|_1+\|g\|_\Gamma$, again $\|\cdot\|_{k}$ are norms on the broken Sobolev spaces.
Moreover, the studies in \cite{bruce1974poisson} of the Poisson problem posed in a domain with intersecting interfaces
suggest that this higher regularity results might not hold in our case.
\end{remark}


\section{Numerical results and discussion}\label{s_numer}
This section collects several numerical examples, which demonstrate the accuracy and capability of our unfitted  finite element method. To verify the convergence rates of the method,  we start with a few examples where exact solution is known. This includes the case of planar intersecting fractures and Darcy flow along curvilinear surfaces.
Further we include an example of a pressure drop driven flow in a more complex network of fractures.

\subsection{A multiple fracture problem with a synthetic solution}
To test the convergence of the method, we first consider the example of an analytically prescribed solution on the two intersecting fractures; \MO{the test is built on an example} from~\cite{brenner2016gradient}. The setup is  given below.


\begin{figure}[ht!]
\begin{center}
\includegraphics[width=0.42\textwidth]{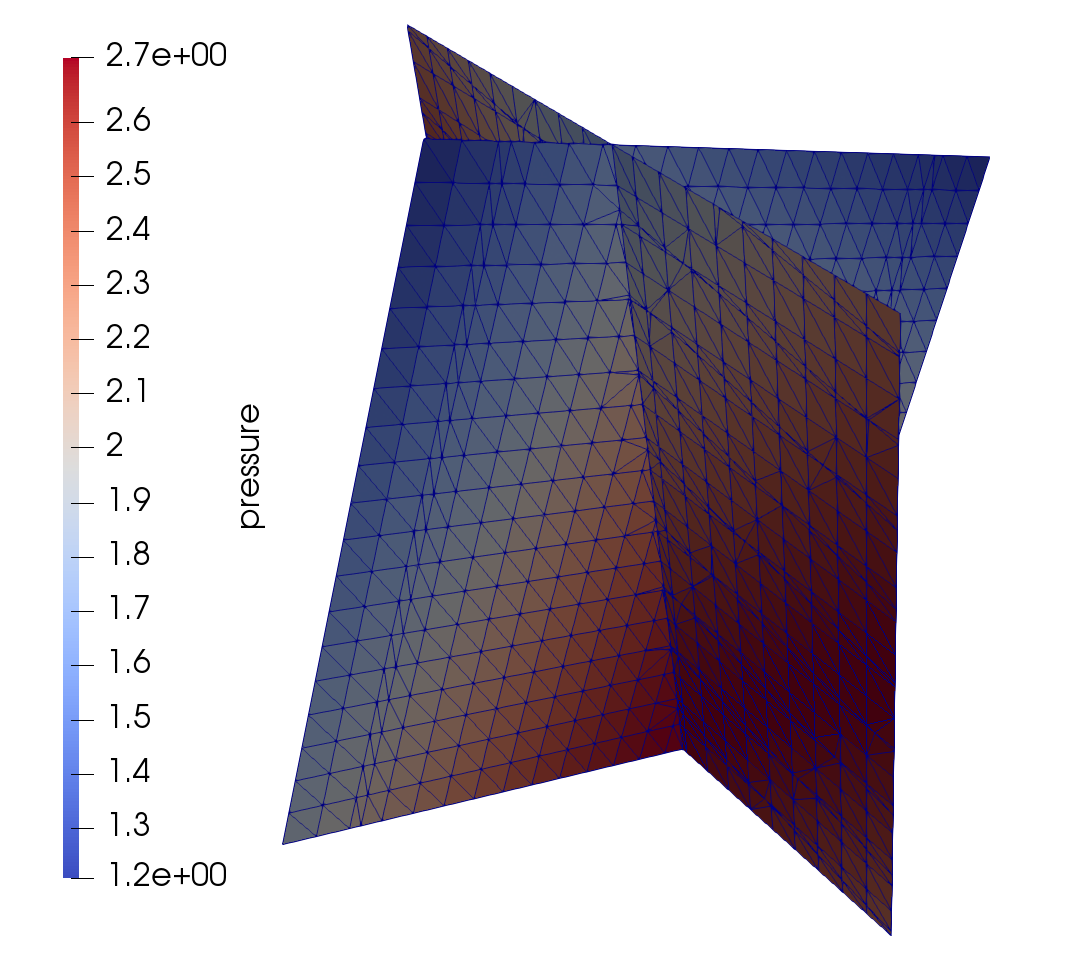}\qquad
\includegraphics[width=0.38\textwidth]{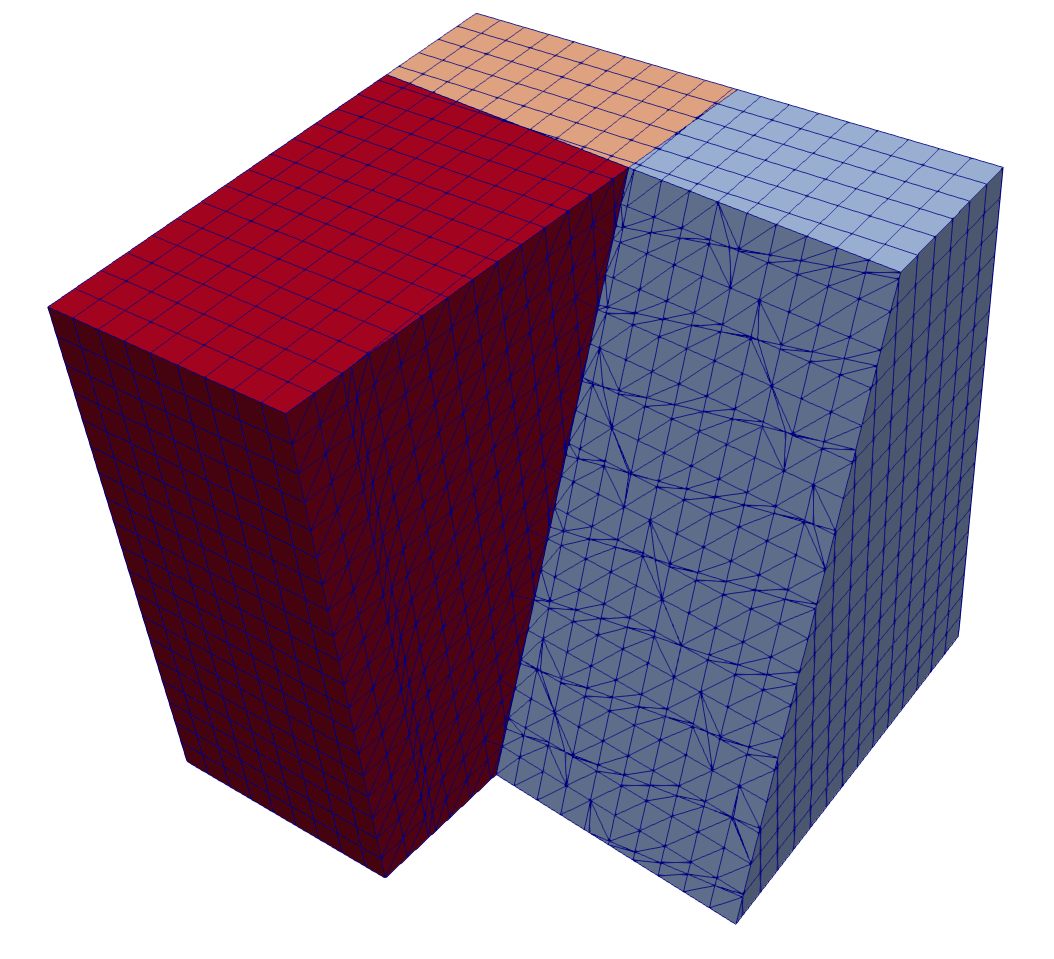}
\caption{\label{fig:rotated_frac2}  (Left) Second level surface mesh and (Right) part of the bulk mesh intersected by the surface for the Example~\ref{exam3} with $\alpha=24^o$, $\beta=4^o$.}
\end{center}
\end{figure}

\begin{figure}[ht!]
\begin{center}

\includegraphics[width=0.32\textwidth]{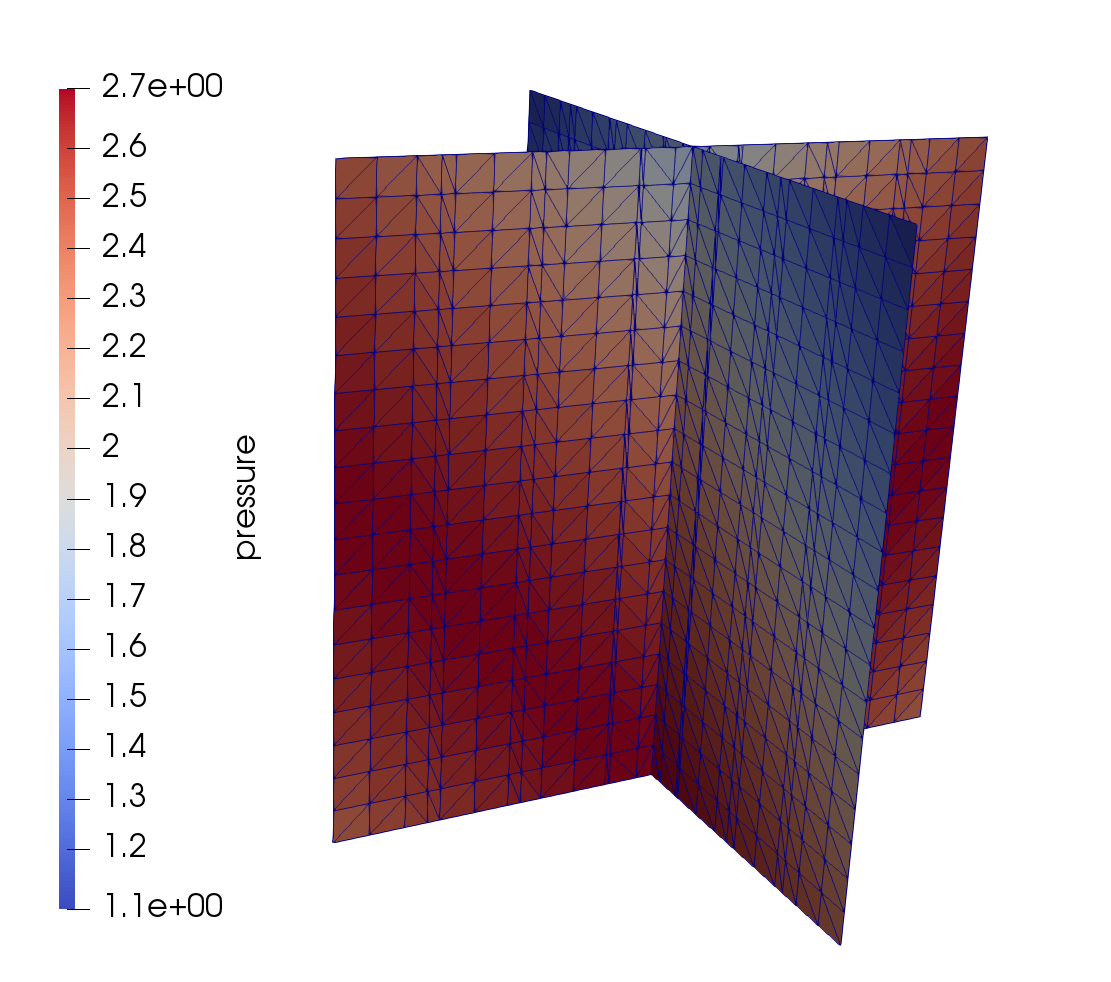}\quad
\includegraphics[width=0.28\textwidth]{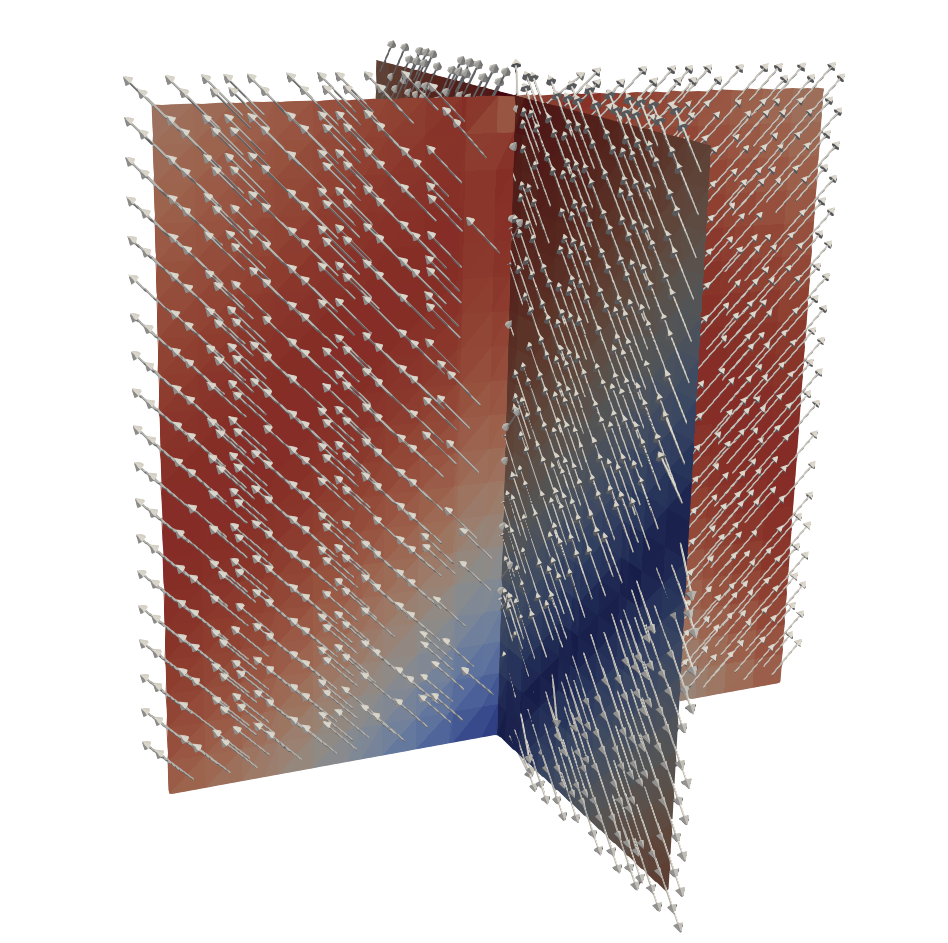}\quad
\includegraphics[width=0.3\textwidth]{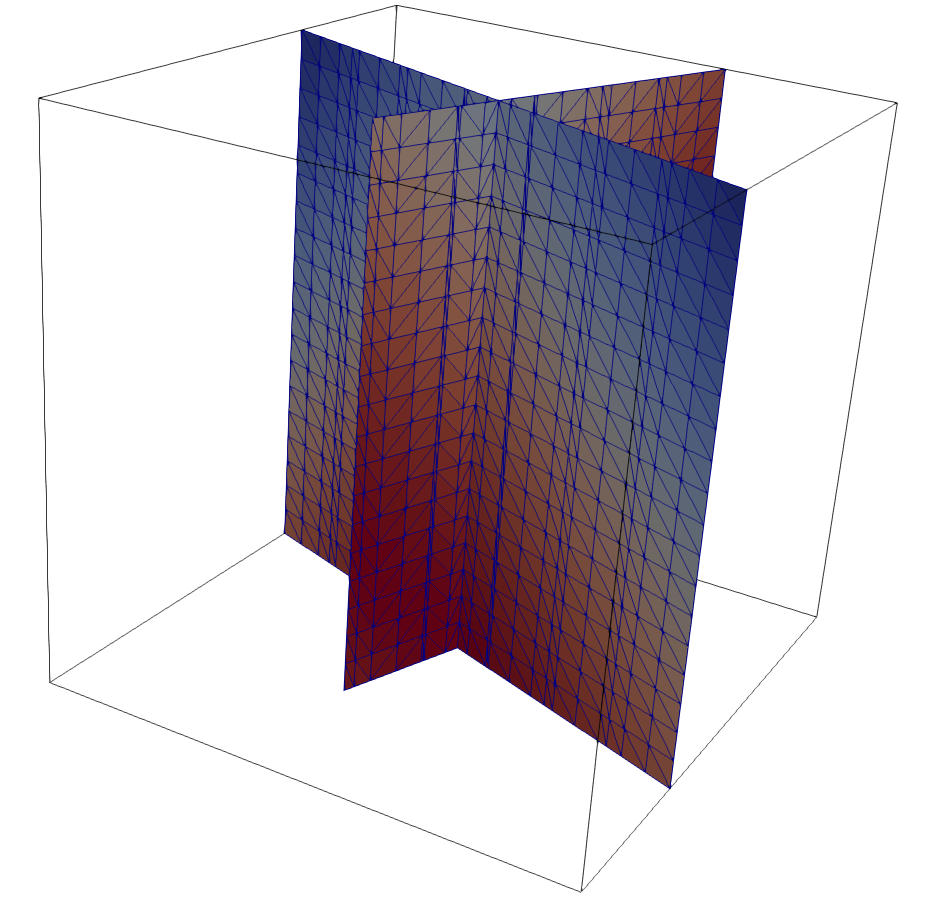}
\caption{\label{fig:rotated_frac}  (Left) The  numerical solution (pressure) and surface mesh from Example~\ref{exam3} with $\alpha=20^o$, $\beta=0^o$. (Center)  The velocity field. (Right) The network and mesh for the case when a part of the fracture's boundary is immersed.}
\end{center}
\end{figure}

\begin{example}\label{exam3}\rm
Consider $\Omega = (0, 1)^3$. \MO{The initial configuration of} the fracture network $\Gamma$ is given by the union of the two rectangles $\{(x, y, z) \in \Omega \mid  x = 0.5\}$ and $\{(x, y, z) \in \Omega \mid  y = 0.5\}$ 
so that $\Gamma=\bigcup\limits_{i=1}^4\Gamma_i$.

To define the exact solution $(\mathbf{u}_i, p_i)$ \MO{for the initial configuration},  we first introduce the functions
$t_1(\mathbf{x}) =  y + z - 0.5$,
$t_2(\mathbf{x}) =  x + z - 0.5$,
$t_3(\mathbf{x}) = -y + z + 0.5$ and
$t_4(\mathbf{x}) = -x + z + 0.5$.
The pressure and the Darcy velocity in each fracture component are given by
\begin{equation}
 p_i(\mathbf{x}) = e^{\cos(t_i(\mathbf{x}))},\quad\text{and}\quad \mathbf{u}_i(\mathbf{x}) = -\sin (t_i(\mathbf{x})) e^{\cos(t_i(\mathbf{x}))} (d_x(i), d_y(i), 1)^T,
\end{equation}
where $i = 1,..,4$ and $d_x = (0, 1, 0, -1), d_y = (1, 0, -1, 0)$.
This $p$ and $\bu$ satisfy \eqref{Darcy}--\eqref{bc} with
 $\mathbf{f} = 0$ and
\[
 g_i = 2 (\cos(t_i(\mathbf{x})) - \sin^2(t_i(\mathbf{x}) ))  e^{\cos(t_i(\mathbf{x}))}.
\]
\end{example}


\begin{table}
\begin{center}
\caption{Errors norms  and convergence rates for Example~\ref{exam3} with $\alpha=20^o$, $\beta=0^o$. \label{tab:brenner1}}\smallskip
\small
\begin{tabular}{r|llllll}\hline
\#d.o.f. & $\|\mathbf{u} - \mathbf{u}_h\|_{L^2}$  & rate & $\|p - p_h\|_{L^2}$ & rate & $\|p - p_h\|_{L^\infty}$& rate \\ \hline\\[-2ex]
384&	  7.606e-2 &      &   4.602e-3 &        & 2.612e-2 & \\
1728&	  3.779e-2 & 1.01 &   1.371e-3 & 1.75   & 1.250e-2 & 1.06\\
7904&     2.081e-2 & 0.86 &   3.925e-4 & 1.80   & 6.118e-3 & 1.03 \\
33072&    1.095e-2 & 0.93 &   1.097e-4 & 1.84   & 3.006e-3  & 1.02 \\  \hline
\end{tabular}
\end{center}
\end{table}

\begin{table}
\begin{center}
\caption{Errors norms  and convergence rates for Example~\ref{exam3} with $\alpha=24^o$, $\beta=4^o$. \label{tab:brenner2}}\smallskip
\small
\begin{tabular}{r|llllll}\hline
\#d.o.f. & $\|\mathbf{u} - \mathbf{u}_h\|_{L^2}$  & rate & $\|p - p_h\|_{L^2}$ & rate & $\|p - p_h\|_{L^\infty}$& rate \\ \hline\\[-2ex]
354&	  9.451e-2 &      &   6.023e-3 &        & 3.738e-2 & \\
1766&	  4.446e-2 & 1.09 &   1.518e-3 & 1.99   & 2.310e-2 & 0.70\\
7320&     1.926e-2 & 1.20 &   2.879e-4 & 2.39   & 1.026e-2 & 1.17 \\
30388&    8.879e-3 & 1.13 &   8.149e-5 & 1.82   & 4.890e-3  & 1.07 \\  \hline
\end{tabular}
\end{center}
\end{table}

\begin{table}
\begin{center}
\caption{Errors norms and convergence rates for Example~\ref{exam3} with $\alpha=20^o$, $\beta=0^o$ and immersed part of the boundary. \label{tab:brenner3}}\smallskip
\small
\begin{tabular}{r|llllll}\hline
\#d.o.f. & $\|\mathbf{u} - \mathbf{u}_h\|_{L^2}$  & rate & $\|p - p_h\|_{L^2}$ & rate & $\|p - p_h\|_{L^\infty}$& rate \\ \hline\\[-2ex]
280&	  7.016e-2 &      &   4.275e-3 &        & 2.459e-2 & \\
1422&	  3.331e-2 & 0.93 &   1.182e-3 &  1.60  & 1.104e-2 & 0.99 \\
6612&     1.821e-2 & 0.88 &   3.419e-4 &  1.79  & 5.805e-3 & 0.92 \\
27924&    9.575e-3 & 0.96 &   9.388e-5 &  1.86  & 3.015e-3  & 0.95 \\  \hline
\end{tabular}
\end{center}
\end{table}

\MO{
To generate less regular intersections of the fractures and the junction line  with the bulk mesh, we next perform the deformation of the fracture system by applying counterclockwise rotations by the angle $\alpha$ about the axis $x = z = 0.5$ and by the angle $\beta$ about the axis $x = y = 0.5$.  The resulting fracture network is denoted by $\Gamma(\alpha, \beta)$. The corresponding change of variables is applied to prescribe the exact solution on $\Gamma(\alpha, \beta)$ using $(\mathbf{u}_i, p_i)$  defined above.
For numerical experiments,  we take $\alpha = 20^o$, $\beta = 0^o$ and $\alpha = 24^o$, $\beta = 4^o$.
}

We next consider a sequence of uniform tessellations of $\Omega$ into cubes  with  $h\in\{1/9,\, 1/19,\, 1/39,\,1/79\}$.
\MO{
The trace of the second level ($h = 1/19$) volumetric grid on $\Gamma(24^o,4^o)$ and the part of the volumetric grid intersected by the surface are illustrated in Fig.~\ref{fig:rotated_frac2}. The computed pressure and velocity for $\Gamma(20^o,0^o)$ are shown in Figure~\ref{fig:rotated_frac}.
Tables~\ref{tab:brenner1}--\ref{tab:brenner2} present the error norms for the computed finite element solutions. We measure the error in the $L^2(\Gamma)$ and $L^\infty(\Gamma)$ for the pressure and $L^2(\Gamma)^3$ for the velocity. 
The results show close to the second order convergence for the pressure $L^2$ norm (although we are unable to prove it) and the first order convergence for the velocity. The  $L^\infty$ norm of the error for the pressure also goes to zero  as $O(h)$.
}

\MO{Finally, we consider the case when a part of the fracture's boundary is immersed in the bulk as illustrated in Figure~\ref{fig:rotated_frac} (right). This corresponds to $\Gamma(20^o,0^o)$, with $\Gamma_2$ cut so that the immersed part of the boundary is vertical and the width of $\Gamma_2$ (i.e. the distance between the immersed boundary in the junction) is $0.25$.
The background mesh does not fit the immersed boundary. Hence, we impose pressure Dirichlet boundary condition using the penalty term as described in Remark~\ref{Rem_Bound}. The finite element error, reported in Table~\ref{tab:brenner3}, appears to be almost unaffected by the presence of the immersed boundary. }


\subsection{Darcy flow over curvilinear surfaces}
We now check if the fracture curvature influences convergence rates of the unfitted finite element method. To this end, we consider Darcy problem \eqref{Darcy}  defined on surface and on the torus.  In both examples, $\Gamma$ is given by closed smooth surfaces (no junctions).

\begin{example}\label{exam1}\rm
We consider $\Gamma=\{ \bx\in\R^3\mid \|\bx\|_2 = 1\}$ embedded in $\Omega= (-2,\,2)^3$.
The solution  $(\mathbf{u}, p)$ is given by
\[
    p(\bx)= \frac{a}{\|\bx\|^3}\left(3x_1^2x_2 - x_2^3\right),\qquad \mathbf{u} = -\nabla_\Gamma p, \quad
    \bx=(x_1,x_2,x_3)\in\Omega,
\]
with $a=12$. One verifies that $\bu$ and $p$ satisfy \eqref{Darcy} with $\mathbf{f} = 0$, and $g=12 p$, so that  $g$ satisfies the compatibility condition $\int_\Gamma g \, \rd \bs = 0$.
\end{example}

\begin{figure}[ht!]
\begin{center}
   \includegraphics[width=0.33\textwidth]{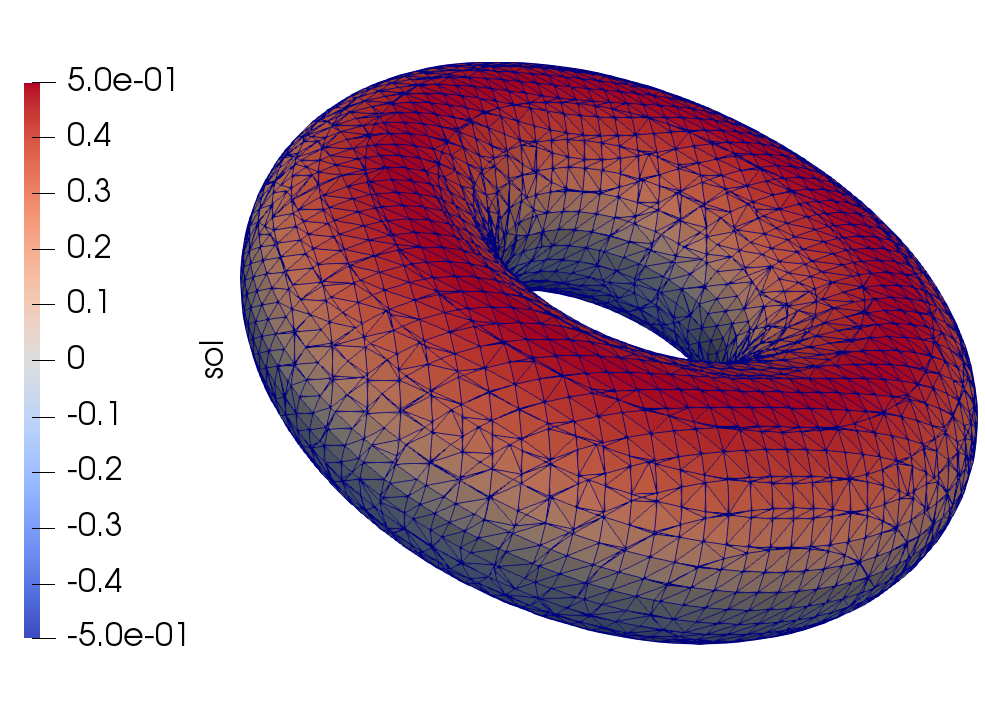}\quad
  \includegraphics[width=0.3\textwidth]{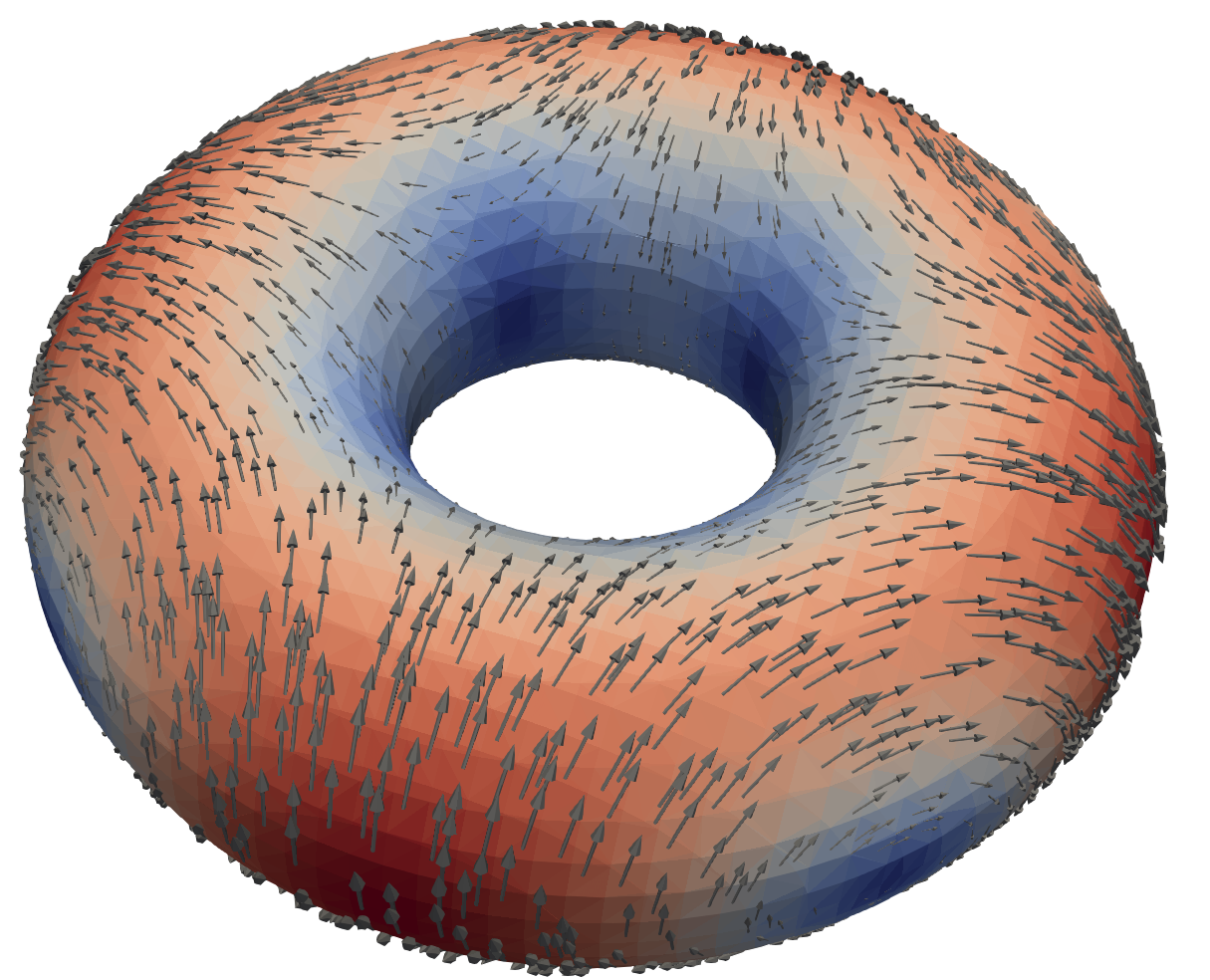}\quad
   \includegraphics[width=0.28\textwidth]{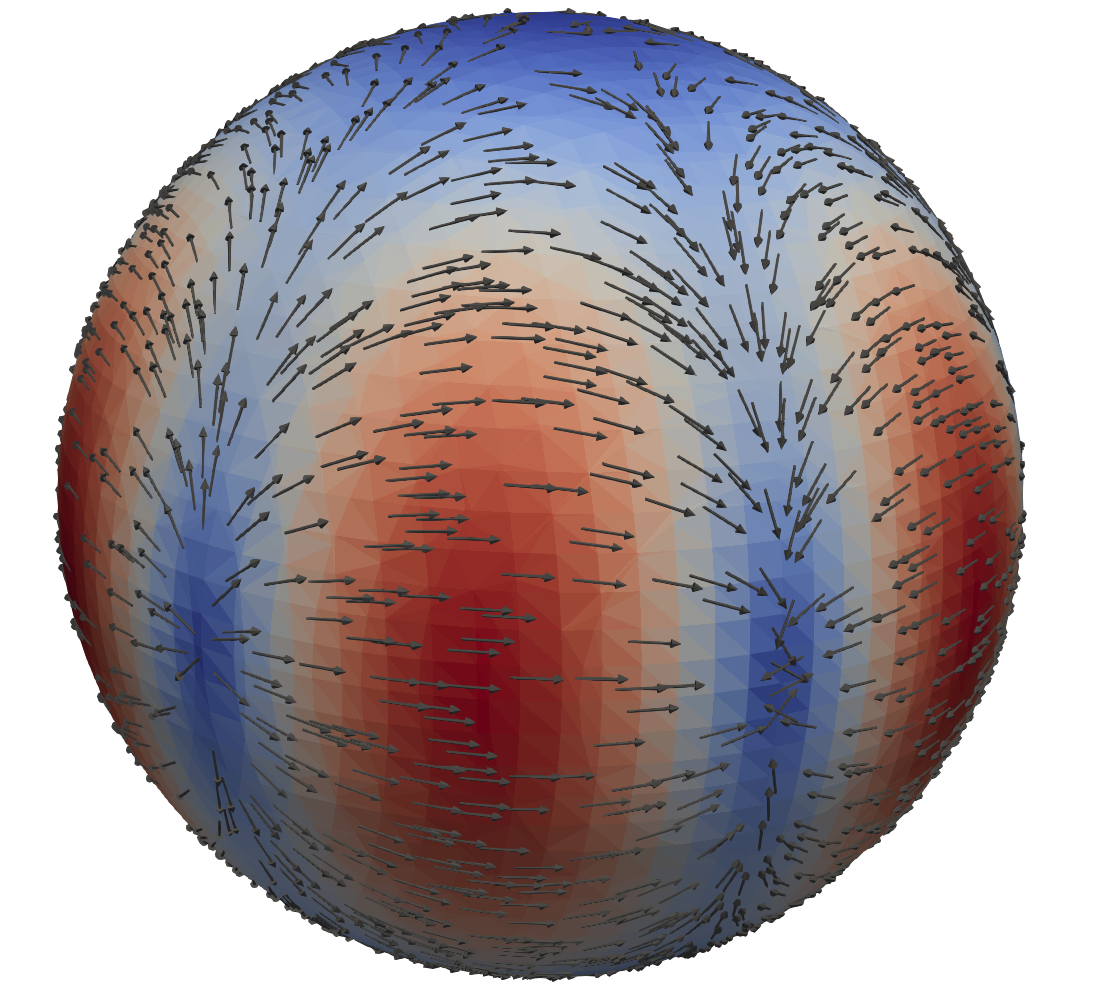}
\caption{\label{fig:sphere_tor} The computed Darcy velocity  from Examples~\ref{exam1} and \ref{exam2}. The left figure shows the pressure and induced surface mesh in Example~ \ref{exam2}.}
\end{center}
\end{figure}

\begin{example}\label{exam2}\rm
This example can be found in~\cite{hansbo2017stabilized} for the Darcy flow along  the torus surface.
We consider $\Gamma = \{ \bx\in\Omega \mid r^2 = x_3^2 + (\sqrt{x_1^2 + x_2^2} - R)^2\}$ embedded in $\Omega=(-1.6, 1.6)\times (-1.6, 1.6) \times (-0.8, 0.8)$. We set $R = 1$ and $r = 0.5$.
The solution $(\mathbf{u}, p)$ to \eqref{Darcy} with right-hand sides $g = 0$ and
\[
\mathbf{f} =
\left(
\begin{aligned}
 x_1x_3(2 - ( 1 - {R}/ \sqrt{x_1^2 + x_2^2})/A)\\
 x_2 x_3(- 2 - ( 1 - {R}/ \sqrt{x_1^2 + x_2^2})/A)\\
1 - \frac{2 (x_1^2 - x_2^2)(\sqrt{x_1^2 + x_2^2} - R)}{\sqrt{x_1^2 + x_2^2}} - x_3^2 / A
 \end{aligned}
\right) \quad \text{with}~ A = (R^2 + x_1^2 + x_2^2 - 2 R \sqrt{x_1^2 + x_2^2} + x_3^2),
\]
is given by
\[
    p(\bx)= x_3,\qquad \mathbf{u} = (2x_1 x_3, -2 x_2 x_3, 2(x_1^2 - x_2^2)(R - \sqrt{x_1^2 + x_2^2})/\sqrt{x_1^2 + x_2^2}).
\]
\end{example}

\begin{table}
\begin{center}
\caption{Errors norms  and convergence rates for the example~\ref{exam1} \label{tab:sphere}}\smallskip
\small
\begin{tabular}{r|llllll}\hline
\#d.o.f. & $\|\mathbf{u} - \mathbf{u}_h\|_{L^2}$  & rate & $\|p - p_h\|_{L^2}$ & rate & $\|p - p_h\|_{L^\infty}$& rate \\ \hline\\[-2ex]
556&	2.250e-0   &      &2.273e-1   &          &6.005e-1 & \\
2332&	5.978e-1   &1.91    &  5.392e-2 & 2.07   &1.593e-1 & 1.91\\
9532&   1.559e-1& 1.94 &   1.372e-2  & 1.97      &4.121e-2 & 1.95 \\
38212&  4.907e-2 & 1.66 &  3.192e-3 & 2.10     &9.613e-3 & 2.10 \\  \hline
\end{tabular}
\end{center}
\end{table}

\begin{table}
\begin{center}
\caption{Errors norms  and convergence rates for the example~\ref{exam2} \label{tab:tor}}\smallskip
\small
\begin{tabular}{r|llllll}\hline
\#d.o.f. & $\|\mathbf{u} - \mathbf{u}_h\|_{L^2}$  & rate & $\|p - p_h\|_{L^2}$ & rate & $\|p - p_h\|_{L^\infty}$& rate \\ \hline\\[-2ex]
560&  6.979e-2&       &1.749e-2  &       &3.762e-2&  \\
2300&  2.042e-2& 1.77 &3.775e-3  & 2.21  &1.016e-2& 1.89 \\
9096&  6.321e-3&2.13  &8.759e-4  & 2.65  &2.328e-3&2.68 \\
35528 & 2.626e-3 &1.60&2.154e-4  & 2.55  &5.933e-4&2.48  \\  \hline
\end{tabular}
\end{center}
\end{table}

Again, we use a sequence of octree bulk grids.
We start with the initial uniform grids with $h = 1/4$ for Example~\ref{exam1} and $h = 8/25$ for Example~\ref{exam2},
which were further gradely refined towards the surfaces.
Tables~\ref{tab:sphere} and~\ref{tab:tor} show finite element errors and convergence rates for the computed solutions over several levels of refinement. We see that the convergence rates overall improve compared to the case with junctions, which is expected from the analysis. At the same time, the fracture bending does not affect the efficiency of the method, which is also well known property of the Trace FEM.
The computed solutions and induced surface meshes are illustrated in Figure~\ref{fig:sphere_tor}. \MO{Convergence rates for this test well agree with those reported in~\cite{hansbo2017stabilized} for tetrahedra bulk mesh.}

\subsection{Pressure drop driven flow in a fracture network}
The last example demonstrates the flexibility in applying the method for the case of more complex fracture networks.
\begin{example}\label{exam4}\rm
We consider a fracture network consisting of $5$ components, both curvilinear and planar,  and embedded in the bulk domain $\Omega = (-1, 1)^3$.
The fracture network is illustrated in Figure~\ref{fig:mix_sol_mesh} (left), where each component $\Gamma_i$, $i=1,\dots,5$ has a distinct color.
On the parts of $\partial\Gamma$ crossing the left and the right sides of the cube, we prescribe the Dirichlet pressure boundary conditions: $p = 2$ for $\{ \mathbf{x} \in \Gamma_1\,:\, x = -1\}$ and $p = 0$ for $\{\mathbf{x} \in \Gamma_2 \cup  \Gamma_4\,:\, x = 1\}$ and for
$\{\mathbf{x} \in \Gamma_5: z = 1\}$.
On the rest of  $\partial\Gamma$ we prescribe  zero-flux conditions. Thus the boundary conditions define the pressure drop that  drives the flow from left to right through the network.
The computed solution is  shown in  Figure~\ref{fig:mix_sol_mesh} (right).

\begin{figure}[ht!]
\begin{center}
\includegraphics[width=0.39\textwidth]{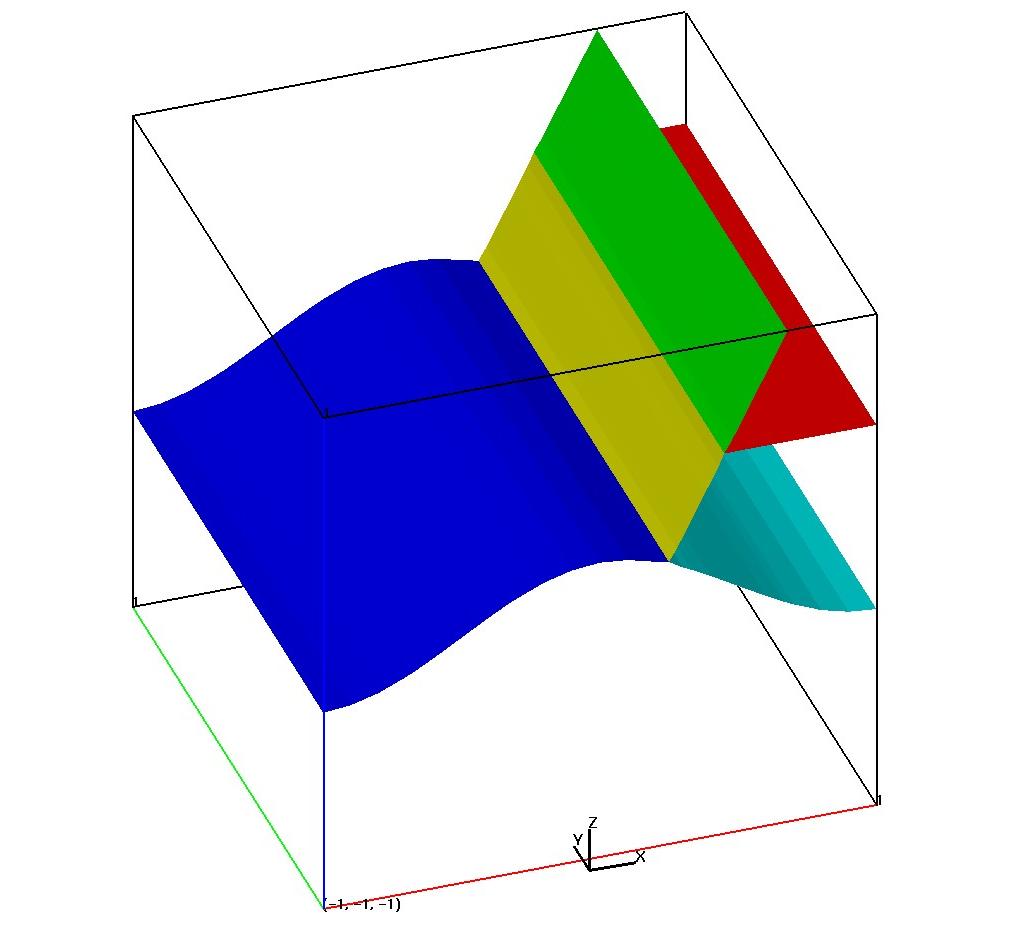}\qquad
\includegraphics[width=0.45\textwidth]{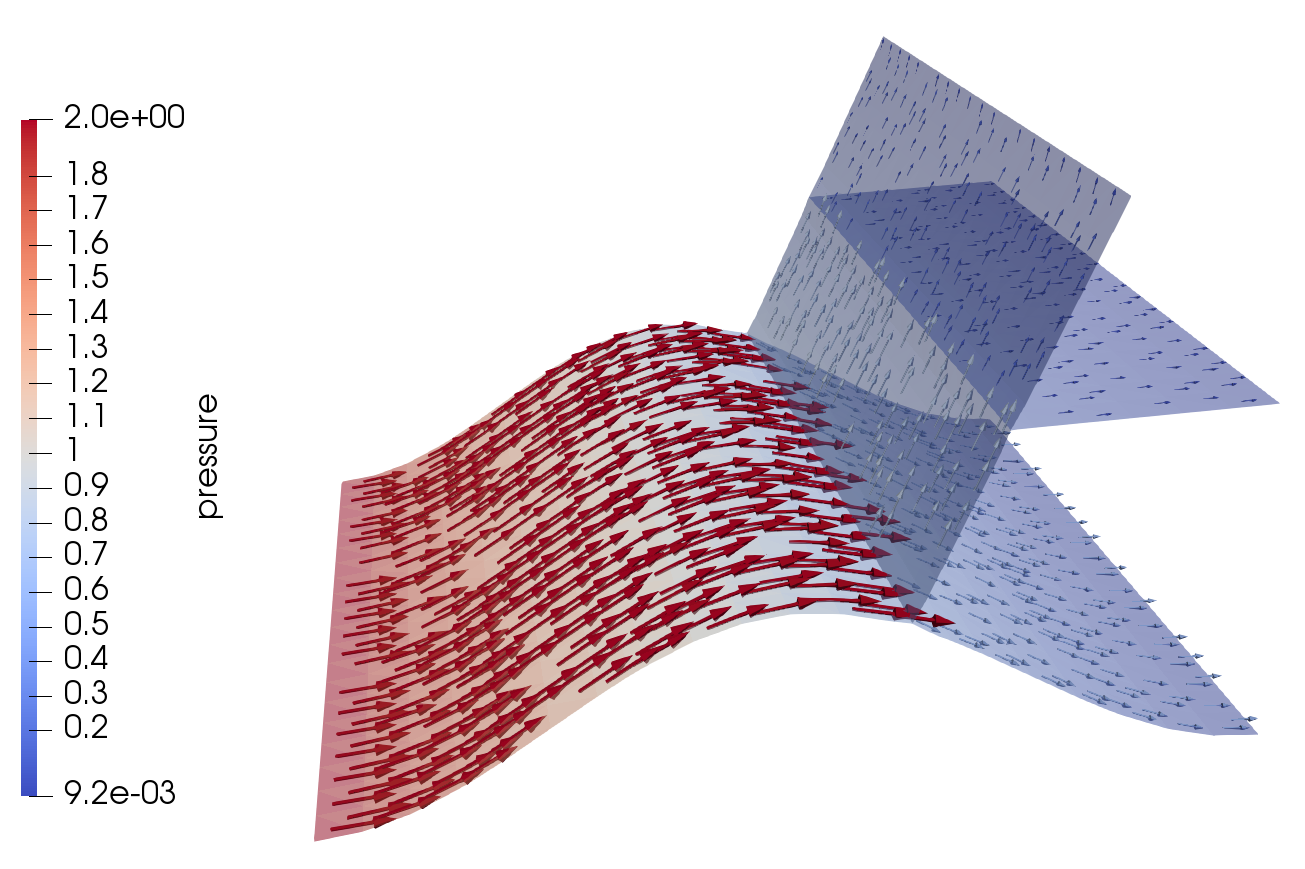}
\caption{\label{fig:mix_sol_mesh}  (Left) Fracture network from Example~\ref{exam4}. (Right) The computed pressure and velocity field}
\end{center}
\end{figure}

\end{example}
\subsection*{Acknowledgments}
The work of the first author (the method development, setup and analysis of numerical experiments) was supported by the Russian Science Foundation Grant 17-71-10173; the second author was supported by the NSF grant~1717516.

\bibliographystyle{abbrv}
\bibliography{ChernOlshan}

\appendix

\section{Proof of the FE trace inequality \eqref{eq_trace2}}

The proof largely follows the arguments given in \cite{guzman2018inf} to prove \eqref{eq_trace} and makes use of the following result found for example in \cite{Grisvard}:
\begin{equation}\label{Trace}
\|v\|_{L^2(\partial\omega)}^2 \le C \|v\|_{H^1(\omega)} \|v\|_{L^2(\omega)}  \quad \text{ for all } v \in H^1(\omega).
\end{equation}
for a bounded domain $\omega\subset\mathbb{R}^n$ with Lipschitz boundary.

\begin{lemma}
Let $T \in \mathcal{T}_h$. There exists an extension operator $R_T: H^2(T) \rightarrow H(\mathbb{R}^3)$ such that
$R_T v= v$ on  $T$ and
\begin{equation}
\|R_T v\|_{L^2(\mathbb{R}^3)} + h_T \|\nabla R_T v\|_{L^2(\mathbb{R}^3)}+ h_T^2 \|D^2 R_T v\|_{L^2(\mathbb{R}^3)} \le   C (\|v\|_{L^2(T)} + h_T \|\nabla v\|_{L^2(T)}+ h_T^2 \|D^2 v\|_{L^2(T)}) \label{extensionT2}
\end{equation}
where the constant $C$ is independent of $T$ and $v$.
\end{lemma}
\begin{proof}   We denote by $\hat{T}$ the reference cube of unit size. The center of $\hat{T}$ is placed at the origin. Then, we know (\cite{stein2016singular}) there exists an extension operator from $R: H^1(\hat{T}) \rightarrow H_0^1(B_{2})$ such that $R \hat{v}=  \hat{v}$ on $\hat{T}$ and
\begin{alignat}{1}
\|R \hat{v}\|_{H^2(B_2)} \le & C \|\hat{v}\|_{H^2(\hat{T})},  \label{refext2}
\end{alignat}
where $B_2$ is the ball with radius 2 centered at the origin.

Let $F_T: \hat{T} \rightarrow T$ be the onto affine mapping and has the form $F_T (\hat{x})= B \hat{x}+ b$. For any $v \in H^2(T)$ we  define $\hat{v} \in H^2(\hat{T})$ in the following way:
$\hat{v}(\hat{x})=v(F_T(\hat{x}))$. Then, the desired extension is given by
$
(R_T v)(x)= (R \hat{v})(F^{-1}_T(x)).
$
For notational convenience we use $w=R_T v$. Then, we see that $\hat{w}=R \hat{v}$. Using a change of variables formula we get
\begin{equation*}
\|D^2 w\|_{L^2(\mathbb{R}^3)}^2= \int_{ F(B_2)}  |D^2 w(x)|^2  dx= \int_{B_2} |B^{-t} D^2 \hat{w}(\hat{x})B^{-1}|^2 |\text{det} B|  d\hat{x}.
\end{equation*}
We have
$
|B_{ij}| \le  C\, h_T $, $
|B_{ij}^{-1} | \le  C \, h_T^{-1}.
$
Therefore, we obtain using \eqref{refext2}
\begin{equation*}
\int_{B_2} |B^{-t} D^2 \hat{w}(\hat{x})B^{-1}|^2 |\text{det} B|  d\hat{x} \le C h_T^{-1} \|D^2 \hat{w} \|_{L^2(B_2)}^2= C h_T^{-1} \|D^2 R \hat{v} \|_{L^2(B_2)}^2\le C h_T^{-1} \|\hat{v} \|_{H^2(\hat{T})}^2.
\end{equation*}
It is standard to show, again using a change of variable formula, and the bounds for $B$ and $B^{-1}$ above that
\begin{equation*}
 h_T^{-1} (\|\hat{v}\|_{L^2(\hat{T})}^2+ \|\nabla \hat{v}\|_{L^2(\hat{T})}^2+  \|D^2 \hat{v}\|_{L^2(\hat{T})}^2) \le C \, (h_T^{-4} \|v\|_{L^2(T)}^2+ h_T^{-2}\|\nabla v\|_{L^2(T)}^2+\|D^2 v\|_{L^2(T)}^2).
\end{equation*}
Therefore, we have shown
\begin{equation*}
h_T^{-1} \|D^2 R_T v\|_{L^2(\mathbb{R}^3)} \le (\|v\|_{L^2(T)}+ h_T\|\nabla v\|_{L^2(T)}+ h_T^2\|D^2 v\|_{L^2(T)}).
\end{equation*}
The required bounds for $\|R_T v\|_{L^2(\mathbb{R}^3)}$ and $\|\nabla R_T v\|_{L^2(\mathbb{R}^3)}$ follow a similar argument; see Lemma~5 in \cite{guzman2018inf}.
\end{proof}

We are now ready to prove the  FE trace inequality \eqref{eq_trace2}.
Let $T \in \mathcal{T}_h$ and let $v \in H^1(T)$. Then, we apply \eqref{Trace} first for $\omega=\Gamma_i$ and next for $\omega=\Omega$, to get (we again use $w=R_T v$ for notation convenience):
\begin{equation*}
\begin{split}
\|v\|_{L^2(T \cap\partial\Gamma_i)} &\le \|w\|_{L^2(\partial\Gamma_i)}
 \le C \|w\|_{L^2(\Gamma_i)}^{1/2}  \| w\|_{H^1(\Gamma_i)}^{1/2} \\
 &\le C \|w\|_{L^2(\Omega\cap\widehat{\Gamma}_i)}^{1/2}  \| w\|_{H^1(\Omega\cap\widehat{\Gamma}_i)}^{1/2}
 \le C ( \|w\|_{L^2(\Omega\cap\widehat{\Gamma}_i)}+ \|w\|_{L^2(\Omega\cap\widehat{\Gamma}_i)}^{1/2}  \|\nabla w\|_{H^1(\Omega\cap\widehat{\Gamma}_i)}^{1/2})\\
 &\le C( \|w\|_{L^2(\Omega)}^{1/2} \|w\|_{H^1(\Omega)}^{1/2}+ \|w\|_{L^2(\Omega)}^{1/4} \|w\|_{H^1(\Omega)}^{1/2} \|\nabla w\|_{H^1(\Omega)}^{1/4})  \\
  &\le C\big ( \|w\|_{L^2(\Omega)}+ \|w\|_{L^2(\Omega)}^{1/2} \|\nabla w\|_{L^2(\Omega)}^{1/2}
  + \|w\|_{L^2(\Omega)}^{1/2} \|\nabla w\|_{L^2(\Omega)}^{1/4}  \|D^2 w\|_{L^2(\Omega)}^{1/4}\\
  &\qquad\quad + \|w\|_{L^2(\Omega)}^{1/4} \|\nabla w\|_{L^2(\Omega)}^{3/4} + \|w\|_{L^2(\Omega)}^{1/4} \|\nabla w\|_{L^2(\Omega)}^{1/2} \|D^2 w\|_{L^2(\Omega)}^{1/4}\big).
\end{split}
\end{equation*}
We apply   Young's  inequality {and use $h_T\le h_0$} to handle terms on the right hand side. For example,
we estimate
\begin{equation*}
\begin{split}
    \|w\|_{L^2(\Omega)}^{1/4} \|\nabla w\|_{L^2(\Omega)}^{3/4} & \le \frac14\|w\|_{L^2(\Omega)}+ \frac34\|\nabla w\|_{L^2(\Omega)}\le C(h_T^{-1}\|w\|_{L^2(\Omega)}+ \frac34\|\nabla w\|_{L^2(\Omega)})\\
     \|w\|_{L^2(\Omega)}^{1/4} \|\nabla w\|_{L^2(\Omega)}^{1/2} \|D^2 w\|_{L^2(\Omega)}^{1/4}& \le
     \frac12\|\nabla w\|_{L^2(\Omega)}+ \frac12\|w\|_{L^2(\Omega)}^{1/2}  \|D^2 w\|_{L^2(\Omega)}^{1/2}\\
     &\le
     \frac12\|\nabla w\|_{L^2(\Omega)}+ \frac1{4h_T}\|w\|_{L^2(\Omega)}  + \frac{h_T}4\|D^2 w\|_{L^2(\Omega)}.
\end{split}
\end{equation*}
Other terms are treated in the same way to get
\begin{equation*}
\|v\|_{L^2(T \cap\partial\Gamma_i)} \le C ( h_T^{-1} \|R_T v\|_{L^2(\Omega)}  +\|\nabla R_T v\|_{L^2(\Omega)}+h_T\|D^2 R_T v\|_{L^2(\Omega)}).
\end{equation*}
The desired result in  \eqref{eq_trace2} now follows after applying  \eqref{extensionT2}.

\end{document}